\documentclass[12pt,a4paper]{article}
\usepackage[utf8]{inputenc}
\usepackage[T1]{fontenc}
\usepackage[english]{babel}
\usepackage{wrapfig}
\usepackage{amsmath}
\usepackage{bbm}
\usepackage{xcolor}
\usepackage{centernot}
\usepackage{amsfonts, enumitem}
\usepackage{amsthm, mathrsfs}
\usepackage{amssymb}
\usepackage{hyperref}
\usepackage{authblk}
\hypersetup{
	colorlinks   = true,    % Colours links instead of ugly boxes
	urlcolor     = blue,    % Colour for external hyperlinks
	linkcolor    = blue,    % Colour of internal links
	citecolor    = red      % Colour of citations
}
\usepackage{graphicx}
\usepackage{tikz}
\usetikzlibrary{datavisualization}
\usetikzlibrary{datavisualization.formats.functions}
\usepackage{gnuplottex}
\usepackage[left=2.5cm, right=2.5cm, top=2.5cm, bottom=2.5cm, bindingoffset=1.cm]{geometry}
\usepackage{fancyhdr} % % %

\title{Critical percolation on preferential attachment graphs with infinite variance}

\author[1]{Peter Mörters}
\author[2]{Lucas Schätze}
\affil[1,2]{\footnotesize Mathematisches Institut, Universität zu Köln, Weyertal 86-90, 50931 Köln, Germany}
{
    \makeatletter
    \renewcommand\AB@affilsepx{: \protect\Affilfont}
    \makeatother

    \affil[ ]{Emails}

    \makeatletter
    \renewcommand\AB@affilsepx{, \protect\Affilfont}
    \makeatother

    \affil[1]{p.moerters@uni-koeln.de}
    \affil[2]{lschaet2@uni-koeln.de}
}
\date{}

\newcommand{\Var}{\operatorname{Var}}

\numberwithin{equation}{section}

\newtheorem{theorem}{Theorem}
\newtheorem{shared}{Dummy}[section]  % Hidden base counter for other environments

\newtheorem{definition}[shared]{Definition}
\newtheorem{lemma}[shared]{Lemma}

\begin{document}

\maketitle

	\begin{abstract} \noindent We study the inhomogeneous random graph with  preferential attachment kernel and degree distribution with power-law exponent $\tau\in(2,3)$ as a representative of the class of graphs of preferential attachment type with infinite variance degrees.  Under bond percolation with a positive retention probability independent of the size~$n$ of the graph there is a unique macroscopic component with high probability. We therefore investigate percolation probabilities $p_n\downarrow0$. We identify a moving critical window at \smash{$p_c \sim \beta n^{_{\frac{\tau-3}{2\tau-2}}}$}. Above this window, when  \smash{$p_n \gg p_c$}, the maximal component has size of order \smash{$n p^{_{\frac{\tau-1}{3-\tau}}}_{_n}$} and it is unique.
    Below this window, when \smash{$n^{_{\frac1{1-\tau}}} \ll p_n \ll  p_c$}, it is non-unique, 
    star-shaped and has size of order	\smash{$n^{_{\frac1{\tau-1}}} p_n$}.
In the critical window itself, the largest component scaled by     \smash{$\sqrt{n}$} converges in distribution to a positive random variable with a law given in terms of a subcritical Norros-Reittu graph. 
%In particular, there is a jump in the size of the largest component at the upper end of the critical window.
%At the upper end of this window the size of the largest component jumps from size of order~$\sqrt{n}$ to a much larger `tiny giant' with size of~order~$n^{\frac{\tau-1}{2}}$. 
This behaviour is markedly different from that seen for other classes of scale-free graphs and is conjectured to persist throughout the broad class of growing graphs  
%{\color{cyan}robust}
%of preferential attachment type 
with infinite variance.
\end{abstract}
   
\setcounter{tocdepth}{2}
\tableofcontents

\section{Introduction}

\subsection{Background}

It is a by now well-established fact in the random graph literature that locally tree-like scale-free random graphs with power-law exponent $\tau$ smaller than three only possess a percolative phase. Despite this absence of a non-degenerate phase transition, Cohen, ben-Avraham and Havlin~\cite{cbh} argue that for models that can locally be approximated by unimodular Galton-Watson trees, percolation critical exponents are
well defined near the limit of extreme dilution and that these exponents bear a strong dependence on~$\tau$. These heuristics have been made rigorous for the configuration model 
%(in a version allowing multiple edges) 
by Dhara, van der Hofstad and van Leeuwarden~\cite{DHL}. They show in particular that
when the degree distribution satisfies a power-law with exponent
$\tau\in(2,3)$ then there is a moving critical window~at
$$p_c\sim \beta n^{\frac{\tau-3}{\tau-1}} \quad\text{ for some $0<\beta<\infty$.}$$
For $p_n\gg p_c$ the largest component is of size $np_n^{\frac1{3-\tau}}$ and for
$p_n\ll p_c$ it is of size \smash{$n^{\frac1{\tau-1}}p_n$}. In the critical window the largest component has size of order \smash{$n^{\frac{\tau-2}{\tau-1}}$}.
%so there is no jump in the order of sizes within or at either end of the critical window.
%i.e.\ when \smash{$p_n\sim \beta n^{\frac{\tau-3}{\tau-1}}$}. 
A unique largest component emerges at the end of the critical window (when $\beta\uparrow\infty$), while components have a trivial star-like structure at the bottom of the critical
window (when $\beta\downarrow0$). Interestingly, this behaviour is partly due to the occurrence of multiple edges in the configuration model.
\smallskip
\pagebreak[3]

Bhamidi, Dhara and  van der Hofstad~\cite{BDH} investigate a class of scale-free graphs with infinite variance and rank-one connection probabilities without multiple edges, including the Norros-Reittu model, Chung-Lu model and inhomogeneous random graphs with a product kernel. Here the critical window is at 
$$p_c\sim \beta n^{\frac{\tau-3}{2}} \quad\text{ for some $0<\beta<\beta_c$.}$$
The characteristic features of this class in this window are
\begin{itemize}
    \item there is a finite  (explicitly computable) $\beta_c>0$ such that for $\beta<\beta_c$ the largest
component has size of an order depending on $\tau$ but strictly smaller than $\sqrt{n}$,
%of order $$n^{\frac{\tau^2-4\tau+5}{2(\tau-1)}}\ll \sqrt{n},$$
    \item scaled by the order this largest component has a random structure related to a preferential attachment-type graph,
    \item above the critical window, when $\beta>\beta_c$, a unique largest component of size~$\sqrt{n}$, the `tiny giant', emerges.
\end{itemize}
%They achieve very fine results on the emergence of a tiny giant component with size of order~$\sqrt{n}$. 
%Here the critical window is at $$p_n\sim \beta n^{\frac{\tau-3}{2}}.$$
%There is an explicitly computable $\beta_c>0$ such that for $\beta<\beta_c$ the largest
%component has size of order $$n^{\frac{\tau^2-4\tau+5}{2(\tau-1)}}\ll \sqrt{n},$$
%and for $\beta>\beta_c$ a unique ‘tiny giant’ of size $\sqrt{n}$ emerges. 
This critical percolation behaviour is conjectured to be universal for rank-one models with infinite variance in the absence of multiple edges.
\smallskip

%{\color{red}Check this again in "Multiscale genesis of a tiny giant for percolation on scale-free random graphs" by Shankar Bhamidi, Souvik Dhara, Remco van der Hofstad. } \smallskip
\pagebreak[3]

\subsection{Description of our results}

We investigate percolation on a class of dynamically growing graphs %preferential  attachment-type graphs 
with infinite variance caused by a preferential  attachment mechanism. 
In this class there is, for all percolation probabilities $p>0$, with high probability, a unique macroscopic component, which comprises $\sim\theta(p)n$ vertices, for some $\theta(p)>0$. However, the 
size of this component decays very quickly when $p\downarrow 0$, namely like
	\begin{align}\label{prop:one}
		\theta(p) = 
        \exp\left(-\frac{c}{p} \, (1+o(1))\right),
%        \exp\left(-\frac{\tau-2+o(1)}{p(\tau-1)}\right),
	\end{align}
for $c=\frac{\tau-2}{\tau-1}$, as shown by the authors in~\cite{waw}. It is therefore interesting to couple the percolation probability $p$ to the graph size $n$ such that $p_n\downarrow0$ and ask for a moving critical window in which we see a qualitative change of behaviour %regarding the size 
in the structure of the largest component. \smallskip

%In our models 
We show that our models exhibit such a critical window at 
$$p_c \sim \beta n^{\frac{\tau-3}{2(\tau-1)}} \quad\text{ for some $0<\beta<\infty$.}$$ 
They are in a new %completely different 
universality class, which has not been uncovered so far, 
%to the best of our knowledge 
even in the extensive non-rigorous statistical physics literature, see the surveys~\cite{DoGo07, Li}. This universality class is characterised by the following defining features,
\begin{itemize}
    \item the largest component has size of order $\sqrt{n}$ throughout the critical window. Note that in all previously known model classes the largest component in the critical window has size of order strictly smaller than $\sqrt{n}$,
    \item scaled by $\sqrt{n}$ the largest component in the critical window has a novel random structure related to a 
    %subcritical Norros-Reittu graph,
    rank-one graph,
    \item above the critical window the largest component is unique and has deterministic asymptotic size of order 
    larger than $\sqrt{n}$ but smaller than $n$, a
    %    $n^{\frac{\tau-1}{2}}$ 
    `tiny giant',
    \item below the critical window the largest components are stars of random size of smaller order
    than $\sqrt{n}$.
\end{itemize}
We conjecture that this  behaviour is universal across 
the broad class of growing graphs 
%{\color{cyan}robust}
with infinite variance degree distribution, which often arise through some preferential attachment mechanism.\smallskip

The model  we investigate is an inhomogeneous random graph with a preferential attachment kernel, which we define in Section~\ref{sec:definition}. It is a natural and relatively simple representative of a 
dynamically growing graph model where preference is given to connections linking young vertices to old ones. Dynamically growing graph models, and in particular preferential attachment type models, tend to have quite different behaviour at the phase boundary compared to the models in the class of the inhomogeneous random graphs with a product kernel studied in \cite{BDH}.
This was explored for graphs with finite variance in a recent paper by Jorritsma, Maillard and M\"orters~\cite{JMM} by means of a detailed analysis of the critical behaviour at the phase transition in growing graph models using branching random walk techniques. Their arguments are particularly useful to investigate subcritical and the most delicate near critical behaviour. %As there is no truely subcritical phase i
In the infinite variance case, in the absence of a nondegenerate phase transition,  we have to rely on more robust techniques, exploring a more classical path counting approach instead.
\smallskip

%Here the critical window arises when $$p_n \sim \beta n^{\frac{\tau-3}{2(\tau-1)}}.$$
We now informally describe our results for inhomogeneous random graphs of preferential attachment type with $\tau\in(2,3)$. These results will be formalized in Theorem~\ref{main} in Section~\ref{2.2}.
\smallskip

Below the critical window, when%
$$
n^{_{\frac1{1-\tau}}} \ll p_n \ll  n^{\frac{\tau-3}{2(\tau-1)}}
$$
the largest component arises from a competition of the initial, most powerful vertices. These vertices are typically not connected and have degree of order \smash{$n^{\frac1{\tau-1}}p_n$} 
and the largest components have the same order and contain at least one of these vertices. For smaller $p_n$ edges are either isolated, when 
$$
n^{-1} \ll p_n \ll  n^{_{\frac1{1-\tau}}}
$$
or even completely absent when $p_n\ll n^{-1}$. In other words, below the critical window the largest components are star-like and  not unique in their size. \smallskip

Throughout the critical window, i.e. when
$$p_n \sim \beta n^{\frac{\tau-3}{2(\tau-1)}},$$ 
the maximal connected component has random size of order $\sqrt{n}$ and is also not unique in this size. Scaled by $\sqrt{n}$ the size of the largest component is random and can be described as the maximum weight component in a subcritical Norros-Reittu graph~$\mathscr G_{\text{NR}}$ with deterministic vertex weights arising from the expected component  size of the vertex in the original graph.
%This graph emerges as a limit from a `core graph', which consists of finitely many early vertices and edges drawn between them if they are connected by paths of length two in the original graph. 
%While the connected components in this graph
%of the core graph 
%have random size, the number of vertices in the original graph connected to a vertex in~$\mathscr G_{NR}$ %the core concentrates,
%under the given scaling, 
%so that the scaling limit of the largest component is a random variable determined by the connection structure of the limit of~$\mathscr G_{NR}$. % the core graph.
\smallskip

Above the critical window the largest component is unique and the size concentrates at a constant multiple of size \smash{$np_{_n}^{_{\frac{\tau-1}{3-\tau}}}$}. A unique largest component of asymptotically constant (but not macroscopic)  size arising above the 
critical window is known as `tiny giant'. Here it emerges with a size larger than $\sqrt{n}$, in contrast to the other known graph models of infinite variance studied in \cite{BDH, DHL}.

\subsection{Overview of the proofs}

Results on the emergence of large components in random graphs are most often based 
on typical behaviour, either globally in terms of high density approximations given by differential equations~\cite{MR, BBS, sen}, or locally given by an exploration process which in the limit takes the form of a random tree~\cite{bjr, dm, nick, JMM}.  Neither of these approaches is promising in the case of infinite variance degree distributions when the connectivity behaviour is governed by a relatively small number of vertices of extreme power. We therefore base our analysis on the identification of scales in which a subgraph of hubs, called the core graph, has good connectivity features. The idea of core graphs has been used so far primarily to study distances in supercritical random graphs, see for example~\cite{mth} and \cite[Section 8.7.1]{rg2}, and in~\cite{BDH} to analyse critical behaviour in rank-one models. We  extend the method to growing graph models with infinite variance, where the structural properties of the core and hence the methods of analysis differ considerably compared to the rank-one case. \smallskip 

%Our definition of a core graph uses two scales, the threshold below which a vertex is a core vertex, and a threshold above which a vertex may serve as a connector between core vertices.  We now present the ideas of the proof in detail, but without explicit reference to the graph model we use, in order to emphasise the general nature of our approach.

We now present the ideas of the proof in detail, but without explicit reference to the graph model we use, in order to emphasise the general nature of our approach. Let $\mathscr G_n$ be the graph with $n$ vertices, indexed in order of increasing arrival time, after bond percolation with retention probability $p_n$. We use two scales to define the \emph{core graph}: Its vertex set consists of the first $m$ vertices, and two such vertices are connected by an edge in the core graph if in $\mathscr G_n$ they are  connected by a path of 
length two  with the index of the middle vertex,  called a \emph{one-connector},
%a weak vertex with 
at least~$sn$. The scales~$m$ and $s$ are chosen
%adapted to the various regimes 
such that the largest component in $\mathscr G_n$ is always given by vertices connected by a path in $\mathscr G_n$ to the largest (or maximal weight) connected component in the core graph. Our proofs are then based on a combinatorial analysis of  the structure of these paths eventually showing that the dominant contribution comes from paths of length one.

\subsubsection{Above the critical window.}

Let
%$m$ of order 
$m=\lfloor \mu n\,p_n^{_{2\frac{\tau-1}{3-\tau}}} \rfloor$
and $s=p_n^{_{\frac{\tau-1}{3-\tau}}}$. The proof relies on variation of the constant~$\mu$ in the definition of $m$. 
Irrespective of this constant, 
%by general results on inhomogeneous random graphs  
the core graph has a giant component~$\mathscr C$ consisting of an asymptotically constant proportion of vertices,
see Lemma~\ref{lem1}. This component is of course much smaller than the largest component in~$\mathscr G_n$, but we show in Lemma~\ref{unique} that, for sufficiently large $\mu$,
the component of~$\mathscr C$  in $\mathscr G_n$ is the unique largest component. We  use path counting arguments to systematically reduce the possible shortest paths leading to~$\mathscr C$ to be essentially just single edges, see Lemmas~\ref{pathlength}, \ref{compofm}, and \ref{orderm}. The size of the largest component is then obtained to leading order as the sum over the degrees in $\mathscr G_n$  of vertices in $\mathscr C$, see Lemma~\ref{uppersuper}. For a core vertex this degree concentrates on the scale \smash{$p{_{_n}^{_{\frac{1-\tau}{3-\tau}}}}$},
see Lemma~\ref{nnc}, and hence the tiny giant has size of order
\smash{$mp^{_{\frac{1-\tau}{3-\tau}}}_{_n}$}, which equals order
\smash{$n\,p^{_{\frac{\tau-1}{3-\tau}}}_{_n},$}
with the exact constant obtained in the limit $\mu\to\infty$.

\subsubsection{In the critical window.}

We choose $m$ of constant order and $sn=m+1$.
%not depending on $n$ but sufficiently large so that 
The largest component in the graph $\mathscr G_n$ restricted to $\{m+1,\ldots,n\}$ is still smaller than the largest degree in  $\mathscr G_n$ with high probability, see Lemma~\ref{lcc}. Hence the largest component overall must contain a vertex in $\{1,\ldots,m\}$. 
%As $n\to\infty$ the number of edges in $\{1,\ldots,m\}$ goes to zero so that the core graph only has edges coming from one-connectors. 
We analyse the connection probabilities between vertices $i,j\in\{1,\dots,m\}$ in the core graph. We show by Poisson approximation arguments that the numbers of one-connectors between vertices $i,j$ in the core graph converge to independent Poisson distributed random variables with expectations proportional to 
\smash{$i^{\frac1{1-\tau}}j^{\frac1{1-\tau}}$}, see Lemma~\ref{indep}.
Hence the core graph, for large $m$, approximates an infinite Norros-Reittu graph, which is subcritical as $\frac1{\tau-1}<\frac12$ and the edge density is small. %Similar to the situation above the critical window w
We then show that the number of vertices in $\{m+1,\ldots,n\}$ connected by a path to a given vertex $j\in\{1,\dots,m\}$ concentrates when scaled by \smash{$p_n n^{\frac1{\tau-1}}$} around its expectation, which is of 
order~\smash{$j^{\frac1{1-\tau}}$},
see Lemma~\ref{3.9}. Showing that these paths do not intersect significantly concludes the proof, in Lemma~\ref{conclusion}.

\subsubsection{Below the critical window.}

Below the critical window we show that, for any $m\in\mathbb N$, the connected components of the vertices in $\{1,\ldots, m\}$ are pairwise disjoint
with high probability, see Lemma~\ref{disj}. These components are essentially the direct neighbours of the vertices and
in the regime when $p_n\gg n^{\frac1{1-\tau}}$ the number of direct neighbours scaled by $n^{\frac1{\tau-1}}p_n$ concentrates at its mean,
as before. Otherwise, this is not the case and the number of direct neighbours is random, see Lemma~\ref{outdeg}. When  $p_n\ll n^{\frac1{1-\tau}}$ a vertex has no more than one direct neighbour.\pagebreak[3]

\section{Statement of the results}

\subsection{Inhomogeneous graphs of preferential attachment type.}\label{sec:definition}

%The inhomogeneous random graph with kernel of preferential attachment type is a solvable model, which shares many features  of more involved preferential attachment models. If set-up with a power-law exponent $\tau\in(2,3)$ the model is \emph{robust} in the sense that there is a connected component comprising an asymptotically positive proportion of the vertices, no matter how small the edge density. However, the asymptotic proportion of vertices in this giant component decreases very quickly to zero when the edge density decreases to zero. Our first result, Theorem~1, identifies  the exact speed at which this happens. In our second and principal result, Theorem~2,  we allow to the edge density to decrease with the graph size and determine the size of the largest component in this 
%case. More specifically, we identify the precise parameter window in which this size of the largest components decreases from linear to subpolynomial size.\bigskip

Inhomogeneous random graphs, as defined by Bollobas, Janson and Riordan in~\cite{rjb}, see also~\cite[Chapter 3.2]{rg2} or \cite{söd}, are parametrized by a symmetric, continuous kernel
$$\kappa \colon (0,1] \times (0,1] \to [0,\infty).$$
Given the kernel, the graph $\mathscr G_n$ has vertex set $\{1,\ldots,n\}$ and any pair of distinct vertices $i,j\in\{1,\ldots,n\}$ is connected by an edge, independently with probability
$$p_{ij}:=\frac1n \kappa \Big( \frac{i}n, \frac{j}n \Big) \wedge 1.$$
In most preferential attachment models the probability at which a vertex arriving at time $j$ connects to an earlier vertex $i$ is proportional to its degree. This degree is of order $(j/i)^\gamma$
for some $0\leq \gamma<1$ 
%, see \cite{dmm}, 
and the proportionality factor therefore inverse to order $\sum_{i=1}^{j-1} (j/i)^\gamma \sim j$. This motivates us to choose the preferential attachment kernel $\kappa$ as
$$\kappa(x,y)=\beta (x\wedge y)^{-\gamma}(x\vee y)^{\gamma-1},$$
for some density parameter $\beta>0$ and parameter $0\leq \gamma<1$. The resulting graphs~$\mathscr G_n$ are called \emph{inhomogeneous random graphs of preferential attachment type}.
{They have essentially the same edge probabilities and a slightly simplified covariance structure compared to the classical preferential attachment models, see \cite{rg1, rg2}. }
For all
%small values of $\beta$ 
$\beta<1$ the connection probabilities of two vertices $i,j$ are explicitly given by
$$p_{ij}= \beta (i\wedge j)^{-\gamma}(i\vee j)^{\gamma-1}.$$
The probabilities $(p_{ij})$ do not depend on $n$
due to the fact that the preferential attachment kernel is homogeneous of index~$-1$, which is a characteristic feature of dynamically growing models and in particular of preferential attachment-type graphs. As a result these graphs can also be defined with infinite vertex set $\mathbb N$, and in this form have been studied by Durrett and Kesten in~\cite{DK}. A version of their model  is featuring in the scaling limits of \cite{BDH}.\smallskip

This model has been studied extensively in the case $\gamma=0$ under the name Dubins' model or uniformly grown graph, see for example~\cite{bjr, JMM}. For $0<\gamma<1$ it is easy to check that the graph thus constructed has a power-law degree distribution with exponent $$\tau=1+\frac1\gamma.$$
%When $0<\gamma<\frac12$ the subcritical behaviour of this model has been studied by M\"orters and Schleicher~\cite{nick} and the behaviour in a window around the critical density by Jorristma et al in \cite{JMM}.
%Note that because the kernel is homogeneous of index $-1$, the connection probabilities $p_{i,j}$ do not depend on $n$. This is due to the dynamic nature of the preferential attachment model.
%\bigskip
%
When $\gamma>\frac12$ the power-law exponent $\tau$ lies
between two and three, a range which is often desired when modelling scale-free networks. Precisely in this case there is no nontrivial phase transition
as the largest component in the graph has macroscopic size no matter how small the edge density, see Dereich and M\"orters~\cite{dm, wll}. We study the inhomogeneous random graphs of preferential attachment type in this regime. Percolating this model with a parameter $p$ is equivalent to replacing $\beta$ by $p\beta$ so that, from now on, we only vary the parameter~$\beta$ and express our results in terms of the model 
parameters~$\beta,\gamma$ instead of $p,\tau$.
%Our first result determines the asymptotic size of this component when the edge density is small.

%\newtheorem{1}{Proposition}[]
%\begin{1} \label{1.1}
% Let $\frac{1}{2}\leq \gamma<1$ and write $C_1(\mathscr{G}_n)$ for the largest connected component. Then we have %as $\beta\downarrow 0$,
%	\begin{align*}
%		\lim_{n\to\infty} \frac{\#C_1(\mathscr{G}_n)}{n} = \exp\left(-(1-\gamma)\frac{1+o(1)}{\beta}\right),
%	\end{align*}
%where the limit is taken in probability, the limiting quantity is deterministic and $o(1)$ refers to the asymptotics~$\beta\downarrow 0$. 
%\end{1}

%{\color{red}\noindent Remark: I deleted the remarks}\\

\subsection{The main result: Behaviour around the critical window.}
\label{2.2}

For any finite graph $G$ we denote by $C_1(G)$ its largest connected component, when either it is clear that it is unique or one can break ties arbitrarily, and by $\#C_1(G)$ its size. 
Similarly, $C_i(G)$ is the $i$th largest component and 
$\#C_i(G)$ its size.\smallskip

Equation~\eqref{prop:one} implies that the largest component $C_1(\mathscr G_n)$ of $\mathscr G_n$ has size of order~$n$ for all $\beta>0$, but if $\beta=\beta(n)$ goes to zero with the number of vertices, the largest component is no longer of macroscopic size. Indeed, by monotonicity we have  $\lim_{n\to\infty}\frac1n \#C_1(\mathscr{G}_n) \leq e^{-\frac{c}{\beta}}$, for all $\beta>0$. We now investigate the size of the largest component in this scenario. \smallskip

Recall that a sequence of nonnegative random variables $(X_n)_{n\in\mathbb N}$ is \emph{tight} if
\begin{equation*}
    \lim_{M\to\infty }\, \sup_{n\in\mathbb N} \,\mathbb P \big( \tfrac{1}{M} \le X_n \le M \big)=1 \;.
\end{equation*}

\begin{theorem} \label{main}
 Let $\tfrac12 <\gamma <1$ 
 %and write $\#C_1(\mathscr{G}_n)$ for the size of the largest  and $\#C_2(\mathscr{G}_n)$ for size of the second largest component in the graph 
 and $\mathscr G_n$ the inhomogeneous random graph with the kernel 
 $\kappa(x,y)=\beta(n) (x\wedge y)^{-\gamma}(x\vee y)^{\gamma-1}$
 for $\beta(n)\searrow0$.
 \begin{enumerate}
     \item[(I)] \emph{Tiny giant regime}: If 
     % $n\beta(n)^{\tfrac{2}{2\gamma-1}}\to\infty$ 
     $\beta(n)n^{\gamma-\tfrac{1}{2}}\to \infty$, then there exists a constant $\theta>0$ 
     %\in\big(0, \,1+\tfrac{1}{\gamma(1-\gamma)}\big)$, 
     such that 
     $$\displaystyle \frac{\#C_1(\mathscr{G}_n)}{n\,\beta(n)^{\frac{1}{2\gamma-1}}}  \overset{\mathbb P}{\longrightarrow } \theta \,\;\text{ and }\;\frac{\#C_2(\mathscr{G}_n)}{n\,\beta(n)^{\frac{1}{2\gamma-1
}}}  \overset{\mathbb P}{\longrightarrow }0 \;.$$
     \item[(II)] \emph{Critical window}: If $\beta(n)n^{\gamma-\tfrac{1}{2}}\to \varphi >0 $, then
     $$ \displaystyle\frac{\# C_1(\mathscr{G}_n)}{\sqrt{n}}  \overset{d}{\longrightarrow} \varphi \sup_{i\in\mathbb N}\bigg\{\sum_{j\in\mathcal C(i)} \frac{1}{\gamma}\, j^{-\gamma}\bigg\}\;,$$
     where $\mathcal{C}(i)$ is the component of $i$ in the random graph
     $\mathscr{G}_{\text{NR}}$ with vertex set $\mathbb N$, where two vertices $i,j\in\mathbb N$ have a \smash{$\text{Pois}(\tfrac{\varphi^2}{2\gamma-1}(ij)^{-\gamma})$} number of edges between them, independently for all pairs of vertices.
     
     \item[(III)] \emph{Star component regime}: If  
    $\beta(n)n^{\gamma-\tfrac{1}{2}}\to 0$  but $\beta(n)n^{\gamma}\to\infty$, 
     %$\limsup_{n\to\infty} \tfrac{\log(\beta(n))}{\log(n)}<\tfrac12-\gamma$ 
     %assume further that $\beta(n)=\theta(n)\,n^{-\alpha},$ with $\alpha\in(\gamma-\tfrac12,\gamma)$ and $\theta(n)$ a function such that $\tfrac{\log\theta(n)}{\log n}\to 0$ as $n\to \infty$. 
     then
     $$
      \frac{\#C_1(\mathscr G_n)}{n^\gamma \beta(n)} \overset{\mathbb P}{\longrightarrow} \frac{1}{\gamma} \;.
     $$
 %     $$\displaystyle \mathbb{P}\bigg( \limsup_{n\to\infty}\frac{\#C_1(\mathscr{G}_n)}{n^{\gamma-\alpha+o(1)}}<\infty\text{ and } \liminf_{n\to\infty}\frac{\#C_1(\mathscr{G}_n)}{n^{\gamma-\alpha+o(1)}}>0\bigg)=1\;.$$
 \item[(IV)]  \emph{Bounded component regime}: 
 \begin{itemize}
     \item[(i)]
 If $\beta(n)n^{\gamma}\to c\in(0,\infty)$, then
     $ \big(\#C_1(\mathscr{G}_n)\big)_{n\in\mathbb N} \text{ is tight}$; 
     \smallskip
  \item[(ii)] if $\beta(n)n^{\gamma}\to 0$ and $\beta(n)n\to\infty$, then
     $ \mathbb P \big(\#C_1(\mathscr{G}_n)= 2\big)\to 1$,\smallskip
   \item[(iii)]  if $\beta(n)n\to0$, then  $ \mathbb P \big(\#C_1(\mathscr{G}_n)=1\big)\to 1$.
     \end{itemize}
 \end{enumerate}
\pagebreak[3]
\end{theorem} 

\noindent\emph{Remarks:}
\begin{itemize}[leftmargin=7mm]
%    \item[(i)] If $\beta(n)n\to0$ then it is easy to see by direct calculation, that the graph has no edges with high probability.
    \item[(i)] In the proof we will see that for fixed $i\in\mathscr G_n$ we have $$\frac{\text{deg}(i)}{n^\gamma\beta(n) } \overset{\mathbb P}{\longrightarrow}\frac1\gamma \, i^{-\gamma}.$$ While the largest component in regime $(I\!I)$ arises from 
    the direct neighbours of a {tight}, random number of vertices, the largest 
    component in regime $(I\!I\!I)$ arises from the direct neighbours of the vertex of 
    maximal degree. It follows that in these regimes  $\#C_2(\mathscr G_n)/ n^\gamma \beta(n)$ does not vanish, contrary to regime~$(I)$.
    \item[(ii)] Comparing the behaviour in the critical window for the rank-one class
    studied in \cite{BDH} and our preferential attachment class there is an interesting duality to be observed. While for a preferential attachment graph the scaling limit of the  largest component arises from a core resembling a Norros-Reittu (or rank one) graph, for the Norros-Reittu random graph (and similar models) the largest component arises from a preferential attachment graph; see Theorem 2.3 and Lemma 3.4 in \cite{BDH}.
    \item[(iii)] When $\gamma <\tfrac12$ there exists a percolation phase transition at a positive density $\beta_c>0$. In this case M\"orters and Schleicher~\cite{nick}  identify the size of the largest subcritical component. Further,
    in \cite{JMM} Jorritsma et al.\ identify a critical window 
    $\beta(n)=\beta_c+ \frac{\alpha}{(\log n)^2}$
    in which the largest component has size of order \smash{$\frac{\sqrt{n}}{\log n}$} if  $\alpha<\pi^2$ and size of larger order if $\alpha\ge \pi^2$.
    \item[(iv)] For $\gamma=\frac12$ the preferential attachment kernel has rank one, but the case of $\tau=3$ and infinite variance is not covered in the papers~\cite{BDH, vdH} discussing such models.
\end{itemize}

\section{Proofs}

\newtheorem{4.1}{Theorem}[section]
In the following we write $x\sim y$ for the event that two vertices $x, y$ are nearest neighbours in the graph $\mathscr G_n$ and for any subset of vertices $\mathcal{C}\subset\mathscr{G}_n$ we write $x\sim \mathcal{C}$ for the event that $x$ is a nearest neighbour of at least one of the vertices in $\mathcal{C}$, similarly $ \mathcal{C}\sim \mathcal{C}'$ for two subsets $\mathcal{C},\mathcal{C}'\subset\mathscr{G}_n$.
In the same spirit we write $x\leftrightarrow y$ if there exists a path of any length connecting $x$ with $y$, i.e. $\exists l\in\mathbb N $ and $x_0,x_1,\ldots x_l\in\{1,\dots,n\}$ with $x_0=x$, $x_l=y$ and $x_{i-1}\sim x_i$ for all $i=1,\dots,l$.

\subsection{Above the critical window}

The shaping structure for the largest connected component both in and above the critical window are the powerful vertices in the core graph, which we now define.

\begin{definition} \label{def1}
Let $m\leq n$ be an integer and $s=s(n)$ such that $m\le sn \le n$. 
  \begin{enumerate}
    \item  Let $x,y \in \{1,\dots,n\}$ be two vertices in $\mathscr{G}_n$. We call $u\in\{ \lfloor sn \rfloor +1,
    \dots,n\}$ a \emph{one-connector} for $x$ and $y$, if $x\sim u$ and $u\sim y$ in $\mathscr{G}_{n}$. If two vertices $x$ and $y$ are distinct and connected by a one-connector we write \smash{$x\overset{_1}{\longleftrightarrow} y$}.
    \item We define the \emph{core graph} by its vertex set $C_m=\{1,\dots,m\}$  and edge set
        \begin{align*}
            E_m:=\big\{\{x,y\}\subset C_m  \,:\, x\overset{1}{\longleftrightarrow} y   \big\}\;.
        \end{align*}
\end{enumerate}
\end{definition}

%        Note that this is not a subgraph of $\mathscr{G}_{2n}$, but every component of $(C_n,E_n)$ corresponds to a larger component in $\mathscr{G}_{2n}$, which includes (but is not limited to) the same vertices as in $C_n$. 

\noindent In the regime  $\beta(n)n^{\gamma-\tfrac{1}{2}}\to \infty$ we look at the core graph with 
\begin{align}\label{mmudef}
    m:=\lfloor \mu \,n\,\beta(n)^{\alpha}\rfloor & \text{ and } s:=\beta(n)^{\alpha /2 } \text{  for } \alpha:=\tfrac{2}{2\gamma-1},
\end{align} 
where $\mu>0$ and denote this graph by $\mathcal G_{n,\mu}^{_{(2)}}$.
%We write $C_i(\mathcal G_{n,\mu}^{_{(2)}})$ for the $i$th largest component of this graph. 
%Note that w
With this definition $m\ll sn $, which will be important below, when we compare core graphs with different values of $\mu$.\smallskip

{The properties of this graph follow from the established theory of inhomogeneous random graphs, which relates the inhomogeneous random graph with kernel 
\begin{align}\label{kappamu}
    \kappa_\mu(x,y)= \mu^{1-2\gamma}(2\gamma-1)^{-1}x^{-\gamma}y^{-\gamma}\,.
\end{align}
to a multitype Galton-Watson tree where the offspring of a vertex of type $x\in(0,1]$ is a Poisson process of types with intensity $\kappa_\mu(x,y) \, dy$, see \cite{rjb, rg1, rg2}. By $\rho(\kappa_\mu,x)$ we denote the survival probability of this multitype Galton-Watson tree when started with a single individual of type~$x$.}

\begin{lemma}\label{lem1} 
        	For any $\mu>0$ the core graph $\mathcal G_{n,\mu}^{_{(2)}}$ satisfies
        $$ \frac{\#C_1\big(\mathcal G_{n,\mu}^{(2)}\big)}{n\,\beta(n)^{\alpha}}\overset{\mathbb P}{\longrightarrow} \mu \int_0^1 \rho(\kappa_\mu, x) \, dx >0\quad \text{ and } \quad \frac{\#C_2\big(\mathcal G_{n,\mu}^{(2)}\big)}{n\,\beta(n)^{\alpha}}\overset{\mathbb P}{\longrightarrow}0.$$  
\end{lemma}

 \begin{proof}
Let $i,j\in C_m$, then
\begin{align*}
    \mathbb P \big( i\overset{1}{\centernot\longleftrightarrow}  j \big) = & \prod_{x=\lfloor sn \rfloor +1}^n 1 - \beta(n)^2i^{-\gamma}j^{-\gamma}x^{2(\gamma-1)}
     \\ = & \, \exp \Big(\sum_{x=\lfloor sn \rfloor +1}^n \log(1-\beta(n)^2i^{-\gamma}j^{-\gamma}x^{2(\gamma-1)})\Big)
    \\ = & \, \exp \Big( - \frac{\beta(n)^2n^{2\gamma-1}}{2\gamma-1}i^{-\gamma}j^{-\gamma}\big(1+o(1)\big)\Big),
\end{align*}
where we estimated the sum by an integral in the last step. Thus
\begin{equation}\label{rankone}
\begin{aligned} 
     \mathbb P \big( i\overset{1}{\longleftrightarrow}  j \big) = & \, 1- \exp\Big(- \frac{\beta(n)^2n^{2\gamma-1}}{2\gamma-1}i^{-\gamma}j^{-\gamma}\big(1+o(1)\big)\Big) 
     \\ = &  1- \exp\Big(- \frac{1}{m}\, \kappa_\mu\Big( \frac{i}{m}, \frac{j}{m}\Big)\big(1+o(1)\big)\Big) \, .
\end{aligned}
\end{equation}
%{\color{red}Der Fehlerterm von oben ist vermutlich nicht additiv sondern multiplikativ.}{\color{blue}stimmt.}
Note that $ \kappa_\mu(x,y)=\psi(x)\psi(y)$, with $\psi(x)=cx^{-\gamma}$ for a suitable  $c>0$. This is the kernel of an inhomogeneous random graph of rank one. 
By \cite[Theorem 3.1]{bjr} this implies that $\mathcal G_{n,\mu}^{_{(2)}}$ has a component of size of order $m$ as
\begin{align*}
    1<||T_{\kappa_\mu}||^2 = \int_0^1 \psi^2(x) \,\text{d}x = \infty,
\end{align*}
where $T_{\kappa_\mu}$ is the Hilbert-Schmidt operator associated with the integral kernel~$\kappa_\mu$
on $L^2((0,1))$. 
%Thus $\mathcal G_n^{(2)}$ has a giant component for every $\mu>0$. 
More precisely,
$$ \frac{\#C_1\big(\mathcal G_{n,\mu}^{(2)}\big)}{m}\overset{\mathbb P}{\longrightarrow} \int_0^1 \rho({\kappa_\mu}, x) \, dx\quad \text{ and } \quad \frac{\#C_2\big(\mathcal G_{n,\mu}^{(2)}\big)}{m}\overset{\mathbb P}{\longrightarrow}0,$$
where $\int_0^1 \rho({\kappa_\mu},x)\, dx>0$ is the survival probability of a branching process associated with the kernel~$\kappa_\mu$ and started in a vertex of uniform random~type.
%By our construction it is clear, that $C_1\big(\mathcal G_n^{(2)}\big)$ is increasing in $\mu$, as we just add more vertices to the graph and the connection probabilities do not depend on $\mu$.
     \end{proof}
     
We now study the size of the component 
of $C_1(\mathcal G_{n,\mu}^{_{(2)}})$ in the graph 
$\mathscr G_n$. For a lower bound we first argue that the scaled sum over the outdegrees of all vertices $v\in \{1,\ldots, m\}$ in $\mathscr G_n$ (i.e.\ the number of connections to younger vertices) concentrate, and then that the overlap between neighbourhoods of vertices in $C_1(\mathcal G_{n,\mu}^{_{(2)}})$ is negligible.

\begin{lemma}\label{nnc}
Let \smash{$\displaystyle\text{\emph{outdeg}}(i):=\sum_{j=i+1}^n \mathbbm{1}_{\{j\sim i\}}$} and $f(x):=\tfrac{1}{\gamma}x^{-\gamma}$. Then,
%Let $m=\lfloor \mu \,n\,\beta(n)^{\alpha}\rfloor$.
for any $\mu>0$ and $\varepsilon>0$, we have 
    \begin{align*}
        \mathbb P \Big(\,\Big|\frac{1}{sn}\sum_{i\in C_1(\mathcal G_{n,\mu}^{_{(2)}})} \text{\emph{outdeg}}(i)-\frac{\mu^{1-\gamma}}{m}\sum_{i\in C_1(\mathcal G_{n,\mu}^{_{(2)}})}f\big(\frac{i}{m}\big) \Big|> \varepsilon \Big) \longrightarrow 0 \,.
    \end{align*}
\end{lemma}
\begin{proof}
   There is $(o_n)$ converging to zero such that, for all $i\le n$,
\begin{align*}
    \mathbb{E}\big[\text{outdeg}(i)\big]=\sum_{j=i+1}^n \beta(n) j^{\gamma-1} i^{-\gamma} = \frac1\gamma \beta(n) n^\gamma i^{-\gamma} +o_n.
\end{align*}
Further, for the variance we have 
\begin{equation}\label{varj}
    \begin{aligned}
    \Var(\text{outdeg}(i)) & = \beta(n)^2\sum_{j=i+1}^n\sum_{l=i+1 \atop l\neq j}^n j^{\gamma-1} l^{\gamma-1}i^{-2\gamma}+ \beta(n)\sum_{j=i+1}^n j^{\gamma-1} i^{-\gamma} \\ & \phantom{xxx} - \bigg( \beta(n) \sum_{j=i+1}^n  j^{\gamma-1}i^{-\gamma}\bigg)^2\\ & = \beta(n)\sum_{j=i+1}^n j^{\gamma-1} i^{-\gamma} -\beta(n)^2 \sum_{j=i+1}^n  j^{2(\gamma-1)} i^{-2\gamma}
    \\ & \le  \frac{1}{\gamma} \beta(n) n^\gamma i^{-\gamma} \;.
\end{aligned}
\end{equation}
Moreover we have
\begin{equation*}
\begin{aligned}
    \frac{1}{sn} \sum_{i\in C_1(\mathcal G_{n,\mu}^{(2)})} \mathbb{E}\big[\text{outdeg}(i)\big] = &\, \frac{1}{n\beta(n)^{\alpha/2}} \sum_{i\in C_1(\mathcal G_{n,\mu}^{(2)})} \frac{1}{{\gamma}}i^{-\gamma}\beta(n) n^{\gamma} +o_n
    \\ = & \, \frac{\mu^{1-\gamma}}{m}\sum_{i\in C_1(\mathcal G_{n,\mu}^{(2)})} f \big(\frac{i}{m }\big)+o_n\,.
\end{aligned}
\end{equation*}
Thus, by Jensen's inequality,
\begin{align*}
       \mathbb E & \Big[\Big|\frac{1}{sn}\sum_{i\in C_1(\mathcal G_{n,\mu}^{_{(2)}})} \text{outdeg}(i)-\frac{\mu^{1-\gamma}}{m}\sum_{i\in C_1(\mathcal G_{n,\mu}^{_{(2)}})}f\big(\frac{i}{m}\big) \Big| \Big]
      \\ & \le  \frac{1}{sn}\sum_{i=1}^m \mathbb E \Big[\big|\text{outdeg}(i)-\mathbb E\big[\text{outdeg}(i)\big]\big|\Big] +o(1)
       \\ & \le  \frac{1}{sn}\sum_{i=1}^m \mathbb E \Big[\big|\text{outdeg}(i)-\mathbb E\big[\text{outdeg}(i)\big]\big|^2\Big] ^{\tfrac{1}{2}} + o(1)
       \\ & \le  \frac{c}{sn} n^{\tfrac{\gamma}{2}}\beta(n)^{\tfrac{1}{2}}m^{1-\tfrac{\gamma}{2}}+o(1) = c \beta(n)^{\tfrac{1}{2(2\gamma-1)}}+o(1) \,.
\end{align*}
The statement now follows by Markov's inequality.
\end{proof}

\begin{lemma}\label{onecon}
We have, for some  $C>0$,
 $$\mathbb E\big[\#\{x\in\{m+1,\dots,n\} \,|\, \exists i,j\le m \,:\,i\sim x\sim j\} \big] \le C \beta(n)^\alpha n.$$
\end{lemma}

\begin{proof}
We have
 \begin{equation*}
    \begin{aligned}
        \mathbb E\big[\#\{x>m \,|\, \exists i,j\le m \,:\,i\sim x\sim j\} \big] \le & \sum_{x=m+1}^n \sum_{i,j=1}^{m} \beta(n)^2 i^{-\gamma}j^{-\gamma}x^{2(\gamma-1)} 
        \\ \le &\, C  \beta(n)^2 \beta(n)^{2\alpha(1-\gamma)}n = C \beta(n)^\alpha n.
    \end{aligned}
 \end{equation*}
 \ \\[-1cm]
 \end{proof}
 \medskip

We now combine these two lemmas with classical results on inhomogeneous random graphs to obtain the precise asymptotics for the number of direct neighbours of $C_1(\mathcal G_{n,\mu}^{_{(2)}})$ in $\mathscr G_n$. 
%Observe that, for every $\mu>0$, this yields a lower bound on the size of the largest component in $\mathscr G_n$.
\smallskip

\begin{lemma}\label{neigh_asymp}
Let  $f(x)=\tfrac{1}{\gamma}x^{-\gamma}$.
Then
\begin{align}\label{pconv}
     \lim_{n\to\infty} \frac{1}{n\beta(n)^{\alpha/2}}\#\big\{x\in\{1,\dots,n\}\,:\, x \sim C_1(\mathcal G_{n,\mu}^{(2)}) \}  = %\overset{\mathbb P}{\longrightarrow} 
     \mu^{1-\gamma} \int_0^1 f(x) \rho(\kappa_\mu , x)\,\text{d}x\,,
\end{align}
in probability, where $\rho(\kappa_\mu , x)$ is the survival probability of the multitype Galton-Watson tree started in $x$ associated with the kernel $$\kappa_\mu(x,y)= \mu^{1-2\gamma}(2\gamma-1)^{-1}x^{-\gamma}y^{-\gamma}\,.$$
\end{lemma}

%which is of order $s n$,
%{\color{red}weil $m \times \beta(n) (n/m)^{\gamma} = n\beta(n)^{1/({2\gamma-1)}}$, das stimmt noch.}

\begin{proof} 
We trivially have
$$ \#\big\{x  \in\{1,\dots,n\}\,:\,  x \sim C_1(\mathcal G_{n,\mu}^{(2)}) \}  \le \sum_{i\in C_1(\mathcal G_{n,\mu}^{(2)})} \text{outdeg}(i)$$
and by Lemma~\ref{onecon} we get, with high probability,
\begin{equation*} %\label{nearneig}
    \begin{aligned}
    \#\big\{x & \in\{1,\dots,n\}\,:\, 
    x \sim C_1(\mathcal G_{n,\mu}^{(2)}) \} 
  \\ \ge \,& \sum_{i\in C_1(\mathcal G_{n,\mu}^{(2)})} \text{outdeg}(i) - m - \#\{x>m \colon \exists i,j\le m \,:\,i\sim x\sim j\}
  \\ \ge \, & \sum_{i\in C_1(\mathcal G_{n,\mu}^{(2)})} \text{outdeg}(i) - o(n\beta(n)^{\alpha/2}).
\end{aligned}
\end{equation*}
% (1-\varepsilon) \tfrac1\gamma \beta(n) n^\gamma \sum_{i\in C_1(\mathcal G_n^{(2)})}  i^{-\gamma} - m - C \beta(n)^\alpha n  \ge \, c \beta(n)^{\alpha/2} n,
It therefore suffices to study the asymptotics of
$$\frac{1}{n\beta(n)^{\alpha/2}} \sum_{i\in C_1(\mathcal G_{n,\mu}^{(2)})} \text{outdeg}(i).$$
By Lemma \ref{nnc}, 
\begin{equation}
\label{nearneig2}
    \frac{1}{sn}\sum_{i\in C_1(\mathcal G_{n,\mu}^{_{(2)}})} \text{outdeg}(i) \,= \,\frac{\mu^{1-\gamma}}{m}\sum_{i\in C_1(\mathcal G_{n,\mu}^{_{(2)}})}f\big(\frac{i}{m}\big) +o(1)\quad \text{in probability.}
\end{equation}
As in Lemma~\ref{lem1} we interpret $C_1(\mathcal G_{n,\mu}^{(2)})$ as the largest connected component of an inhomogeneous random graph with vertex set $\{1,\dots,m\}$, where vertex $i$ has type $x_i=i/m$ and kernel given in (\ref{kappamu}). Furthermore
$$  \frac{1}{m}\sum_{i=1}^m f \big(x_i\big) \to \int_0^1 f(x) \,\text{d}x<\infty.$$
Thus \cite[Theorem 9.10]{rjb} is applicable and yields
\begin{align*}
    \frac{1}{m}\sum_{i\in C_1(\mathcal G_{n,\mu}^{(2)})} f \big( x_i)  {\longrightarrow} \int_0^1 f(x) \rho(\kappa_\mu , x)\,\text{d}x\qquad \text{ in probability.}
\end{align*}
Plugging this into (\ref{nearneig2})  gives us the statement.
\end{proof}

Observe that, for every $\mu>0$, the previous lemma yields a lower bound on the size of the largest component in $\mathscr G_n$. As this bound is increasing in $\mu$ we obtain the optimal lower bound in the limit $\mu\to\infty$. 

\begin{lemma}\label{mu}
    We have
    $$ \lim_{\mu\to\infty} \mu^{1-\gamma} \int_0^1 f(x) \rho(\kappa_\mu , x)\,\text{d}x =: \theta >0\,.$$
\end{lemma}

\begin{proof}
   The existence of the limit follows from the monotonicity of the expression in~\eqref{pconv} as a function of $\mu$. The kernel $\kappa_\mu$ is of rank one and can therefore be expressed as 
   $\kappa_\mu(x,y) = c \psi(x)\psi(y)$
   where 
   $$c= \frac{\mu^{1-2\gamma}}{2\gamma-1}\;\text{ and }\;\psi(x)=x^{-\gamma}\,.$$
   Note that $f/\psi=1/\gamma$ is just a constant. Thus we can write 
   \begin{align} \label{feqpsi}   
   \mu^{1-\gamma} \int_0^1 f(x) \rho(\kappa_\mu , x)\,\text{d}x =\frac{\mu^{1-\gamma}}{c\gamma} \underbrace{c\int_0^1 \psi(x) \rho(\kappa_\mu , x)\,\text{d}x}_{=:\rho(c)}\,. \end{align}
   The expression $\rho(c)$ is closely related to the size of the giant component of the rank-one graph and agrees with \cite[Equation (16.12)]{rjb}, where the asymptotic behaviour as $c\to 0$ is discussed. In our case this corresponds to the limit $\mu\to\infty$. The, unfortunately unnumbered, equation after (16.17) in \cite{rjb} gives, for some $C_1>0$, that
   \begin{align}\label{rhoc}
       \rho(c) \sim C_1 c^{\tfrac{1}{2-1/\gamma}} \text{as } c\to 0\,.
   \end{align}
   Plugging this into (\ref{feqpsi}) we find 
   \begin{align*}
        \mu^{1-\gamma} \int_0^1 f(x) \rho(\kappa_\mu , x)\,\text{d}x\sim\,&\frac{\mu^{1-\gamma}}{c\gamma}C_1 c^{\tfrac{1}{2-1/\gamma}}
         =: \theta ,
        %\mu^\gamma \mu^{\tfrac{1-2\gamma}{2-1/\gamma}} = C_2\,.
   \end{align*}
   where $\theta\in(0,\infty)$ depends on $\gamma$, but not on $\mu$.
\end{proof}

Summarizing, $\theta n \beta(n)^{\alpha/2}$ as defined in Lemma~\ref{mu} is an asymptotic lower bound on the size of the largest component in $\mathscr G_n$. 
%\smallskip
We now look at the component in $\mathscr G_n$ of the largest component in the core graph $\mathcal G_{n,\mu}^{_{(2)}}$. We will show in Lemma~\ref{uppersuper} that its size matches this lower bound asymptotically for large $\mu$, in particular it has size of order $n \beta(n)^{\alpha/2}$. This proof requires several steps of preparation.%
\smallskip

 We keep $\mu$ and thus $m$ fixed throughout the proof and first analyze the structure of paths that connect vertices to the core graph $\mathcal G_{n,\mu}^{_{(2)}}$. We first show in Lemma~\ref{pathlength} that there is a positive integer $k$, depending only on $\gamma$, such that the set of vertices connected to the core by a
\begin{itemize}
    \item shortest path longer than $k$, or
    \item a shortest path of even length,
\end{itemize} has size of smaller order than $n\beta(n)^{\alpha/2}$.  For the proof we recall the following technical tool from \cite[Lemma 1]{waw}. Note that this was also proved earlier in \cite{dmm}, but without taking care of the exact constants, which in our application are relevant as they depend on $\beta$ and therefore on~$n$. 

    \begin{lemma} \label{1.3}
		Let $\mu \colon\{1,\dots,n\}\to[0,\infty)$ be a function satisfying, for some $\psi, \phi >0$ and $s_1\in(0,1)$, that
		\begin{align*}
			\mu(x)\le \mathbbm{1}_{\{x\ge s_1 n\}}\psi x^{\gamma-1}+\phi x^{-\gamma}
            \qquad
		\text{for all $x\in\{1,\dots,n\}$. }
        \end{align*}
        Then, for any $s_2\le s_1 $,  we have that
		\begin{align*}
			\sum_{s_2 n<y\le n}\mu(y) p_{xy} & \le \mathbbm{1}_{\{x >s_2 n\}}\beta(n) \Big(\phi \,\frac{(s_2 n)^{1-2\gamma}}{2 \gamma -1}+\psi\log\left(\tfrac{1}{s_2}\right)\Big) x^{\gamma-1}\\& \phantom{xxxx} +\beta(n) \Big(\phi \log\left(\tfrac{1}{s_2}\right)+\psi\,\frac{n^{2 \gamma -1}}{2 \gamma -1}\Big) x^{-\gamma}.
		\end{align*}
	\end{lemma}

\begin{lemma}\label{pathlength}
Let
$\mathcal{N}_k(\mathcal G_{n,\mu}^{_{(2)}})$ the set of vertices $x\in\{m+1,\ldots,n\}$, which are connected to $\mathcal G_{n,\mu}^{_{(2)}}$ by a shortest path of length at least $k$.  Then for any $k\ge\big\lfloor\tfrac{1-\gamma}{2\gamma-1}\big\rfloor +1  $ we have 
$$\frac{\#\mathcal{N}_{2k+1}(\mathcal G_{n,\mu}^{(2)})}{n\beta(n)^{\alpha/2}}\overset{\mathbb P}{\longrightarrow }0\qquad \text{ in probability.}$$
Also, all vertices at even distances are negligible, that is
$$\frac{1}{n\beta(n)^{\alpha/2}}\sum_{k=1}^\infty \#\big(\mathcal{N}_{2k}(\mathcal G_{n,\mu}^{(2)})\setminus \mathcal{N}_{2k+1}(\mathcal G_{n,\mu}^{(2)})\big)
\longrightarrow 0\qquad \text{ in probability.}$$
In particular if $\gamma>\tfrac23$, then
$$\frac{\#\mathcal{N}_2(\mathcal G_{n,\mu}^{(2)})}{n\beta(n)^{\alpha/2}}\longrightarrow 0\qquad \text{ in probability.}$$
    
\end{lemma}
\begin{proof}[Proof of Lemma~\ref{pathlength}]
The proof uses a path counting argument. Recall $m, s$ from \eqref{mmudef} %Let $m=\lfloor \mu \,n\,\beta(n)^{\alpha}\rfloor$, 
and let $\tilde s=\tilde s(n)$ to be determined, but such that $m \le \tilde sn \ll sn $.
Let $X_l$ be the number of paths $x_0x_1\cdots x_{l-1}x_l$ of length $l$, which satisfy the following conditions 
\begin{enumerate}
    \item[(a)] If $l=2$, then $x_0,x_1>m$ and $x_2 \le m$.
    \item[(b)] If $l\ge 3$, then $x_k>\tilde sn$ for all $k<l$ and $x_l\le \tilde sn$. 
\end{enumerate}
Then we have the following upper bound 
\begin{align} \label{13}
   \frac{1}{sn}\#\mathcal{N}_k(\mathcal G_{n,\mu}^{(2)}) \le    \frac{1}{sn}\sum_{l\ge k} X_l + o(1),
\end{align}
using that the total number of vertices in $\{m+1,\ldots, \tilde{s}n\}$ is $o(sn)$.
It remains to show that the other terms in (\ref{13}) go to zero in probability. We have
\begin{align*}
    \mathbb{E}[X_2]=&\sum_{i=1}^m \sum_{x,y=m+1}^n \beta(n)^2 i^{-\gamma}x^{\gamma-1}(x\wedge y)^{-\gamma}(x\vee y)^{\gamma-1}
    \\ \le & \sum_{i=1}^m \sum_{x,y=m+1}^n \beta(n)^2 i^{-\gamma}\big(x^{-1}y^{\gamma-1}+y^{-\gamma}x^{2(\gamma-1)}\big)
   \\ \le &  C\beta(n)^2 m^{1-\gamma}\big(\log(\tfrac{n}{m})n^{\gamma}+n^{\gamma}\big)
    \le   C n \beta(n)^{\gamma\alpha}\log(\beta(n)^{-\alpha}) = o(sn)\,,
\end{align*}
as $\gamma>\tfrac12$. 
For $l\ge 3$ we look at the expectation
\begin{align*}
	     \mathbb{E}[X_l]&\le \sum_{0<x_l\le \tilde sn} \; \sum_{\tilde sn<x_{l-1}\le n} \cdots \sum_{\tilde sn<x_{0}\le n} \prod_{i=0}^{l-1} \beta(n)(x_i \vee x_{i+1})^{\gamma-1}(x_i \wedge x_{i+1})^{-\gamma} \; .
	\end{align*}
	Note that in the above we would actually get an equality if we restrict the summation variables to be distinct.
	We define, for all $k\in\{1,\dots,l\}$,
	\begin{align*}
		\mu_k(x_k):=\sum_{\tilde sn<x_{k-1}\le n} \cdots \sum_{\tilde sn<x_{0}\le n} \prod_{i=0}^{k-1} \, \beta(n)(x_i \vee x_{i+1})^{\gamma-1}(x_i \wedge x_{i+1})^{-\gamma}.
	\end{align*}
	This is a recursion, since
\begin{align} \label{recurmu}
  \mu_{k+1}(x)=\sum_{\tilde sn<y\le n}\mu_k(y) \beta(n) (x \vee y)^{\gamma-1}(x \wedge y)^{-\gamma} \quad \text{ for all } k\ge1 \; .
\end{align}
Thus we can write for the expectation
\begin{align} \label{14}
		\mathbb{E}[X_l]&\le \sum_{0<x_l\le \tilde sn} \mu_l(x_l)
\end{align}
and apply Lemma \ref{1.3}. For the involved variables  we choose $s_1=s_2=\tilde s =\beta(n)^{r}$, with $r\in(\tfrac{1}{2\gamma-1},\tfrac{2}{2\gamma-1}]$. This guarantees $m\le \tilde s n \ll sn$.
 First we have for the starting value of our recursion
\begin{align*}
    \mu_1(x_1)= & \sum_{\tilde sn<x_{0}\le n} \beta(n)(x_0 \vee x_{1})^{\gamma-1}(x_0 \wedge x_{1})^{-\gamma}
		\\ \le & \mathbbm{1}_{\{x_1 >\tilde sn\}}\frac{\beta(n) n^{1-\gamma}}{1-\gamma}\, x_{1}^{\gamma-1}+\frac{\beta(n) n^{\gamma}}{\gamma}\, x_{1}^{-\gamma} \;.
\end{align*}
Thus we define $\psi_1=\tfrac{\beta(n) n^{1-\gamma}}{1-\gamma}$ and $\phi_1=\tfrac{\beta(n) n^{\gamma}}{\gamma}$. Then, using Lemma \ref{1.3}, we define recursively for all $k\ge1$,
\begin{equation}
    \label{psiphi}
\begin{aligned}
		\psi_{k+1}&=\phi_k  \beta(n)\frac{(\tilde s n)^{1-2\gamma}}{2\gamma-1}+\psi_k \beta(n) \log\left(\tfrac{1}{\tilde s}\right),
		\\ \phi_{k+1}&=\phi_k \beta(n) \log\left(\tfrac{1}{\tilde s}\right)+\psi_k \beta(n) \frac{n^{2 \gamma -1}}{2\gamma-1} \;.
\end{aligned}
\end{equation}
From this definition one can see immediately
\begin{align*}
    \psi_2 =O\Big(\beta(n)^{2}\tilde s ^{1-2\gamma} n^{1-\gamma}  \Big) \quad \text{ and } \quad
    \phi_2 =O\Big( \beta(n)^2  \log\left(\tfrac{1}{\tilde s}\right) n^{\gamma}  \Big) \;.
\end{align*}
However, one should calculate the first few terms. We give one more
\begin{align*}
    \psi_3 =O\Big(\beta(n)^{3}\tilde s ^{1-2\gamma}\log\left(\tfrac{1}{\tilde s}\right) n^{1-\gamma}  \Big) \quad \text{ and } \quad
    \phi_3 =O\Big( \beta(n)^{3}\tilde s ^{1-2\gamma}   n^{\gamma}  \Big)  \;.
\end{align*}
One can see by induction that, for all $k\ge1$,
\begin{equation}
\begin{aligned} \label{15}
   \psi_k =&\,O\Big(\beta(n)^{k}\tilde s ^{\lfloor \tfrac{k}{2}\rfloor(1-2\gamma)} \log\left(\tfrac{1}{\tilde s}\right)^{\mathbbm{1}_{\{\text{k odd}\}}} n^{1-\gamma} \Big), \\
    \phi_k =&\,O\Big(\beta(n)^{k}\tilde s ^{( \lceil \tfrac{k}{2}\rceil -1)(1-2\gamma)} \log\left(\tfrac{1}{\tilde s}\right)^{\mathbbm{1}_{\{\text{k even}\}}} n^{\gamma}  \Big) .
\end{aligned}    
\end{equation}
Observe that $\phi_k=o(\phi_{k-1})$ if $k$ is even. 
%Thus the sum over all path-lengths reduces to the sum over the paths of uneven length. {\color{red}Say explicitly what this means.}{\color{blue}this is exactly what happens below} Now b
By (\ref{14}) and Lemma \ref{1.3} we get
\begin{align*}
    \sum_{l=3}^n \mathbb{E}[X_l] &\le \sum_{l=1}^\infty 2\sum_{0<x_{2l+1}\le \tilde s n} \phi_{2l+1} x_{2l+1}^{-\gamma}
    \le c n  \tilde s^{1-\gamma}\beta(n)^3 \tilde s ^{1-2\gamma}\sum_{l=0}^\infty \big( \beta(n)^{2 }\tilde s ^{(1-2\gamma)}\big)^l\,.
\end{align*}
Using that $\tilde s = \beta(n)^r$ we see that the infinite sum above converges if and only if $r<\alpha$, as %should be 
expected. If $\gamma>\tfrac23$ we can choose $r>\tfrac{1}{2\gamma-1}$, such that the term above is~$o(sn)$. For the sake of simplicity let \smash{$r=\tfrac{4}{3}\tfrac{1}{2\gamma-1}$}, then 
\begin{align*}
     \sum_{l=3}^n \mathbb{E}[X_l] &\le  \,C n \beta(n)^{(2\gamma-\tfrac13)\tfrac{1}{2\gamma-1}} = o(sn)\,,
\end{align*}
since $2\gamma-\tfrac13>1$ if $\gamma>\tfrac23$. Now if $\gamma\in(\tfrac12,\tfrac23]$ we can still reduce the $k$-neighbourhood of $\mathcal G^{_{(2)}}_{n,\mu}$ to certain finite $k$. Fix $r=\tfrac{1+\varepsilon}{2\gamma-1}$ for some small $\varepsilon>0$ and $k\ge1$. Then, similar to before, we have 
\begin{align*}
     \sum_{l=2k+1}^n \mathbb{E}[X_l] &\le \sum_{l=k}^\infty 2\sum_{0<x_{2l+1}\le \tilde s n} \phi_{2l+1} x_{2l+1}^{-\gamma}
     \le C n  \beta(n)^{2k+1} \tilde s ^{k(1-2\gamma)+1-\gamma} \,,
\end{align*}
which is $o(sn)$ if %and only if 
%\begin{align*}
    $2k+1+r(1-\gamma)+k(1-2\gamma)r>\tfrac{1}{2\gamma-1}$,
%  \\
%    \iff \quad (1-\varepsilon)k>&\tfrac{(1-\varepsilon)(1-\gamma)}{2\gamma-1}  \iff \quad 
and hence for $k\ge \lfloor\frac{1-\gamma}{2\gamma-1}\rfloor +1$.% 
\smallskip%

 Finally, we look at the vertices with an even distance to $\mathcal G^{(2)}_{n,\mu}$. Recall that the number of vertices at distance two or at distance greater than \smash{$K_\gamma := 2 \lfloor\frac{1-\gamma}{2\gamma-1}\rfloor+2$} is of order $o(sn)$. %Hence we only need to consider distances four to $K_\gamma$. If we c
 Choose \smash{$\tilde s = \mu  \beta(n)^{_{\frac{2}{2\gamma-1}}}$} so that $\lfloor\tilde s n \rfloor= m$. Using 
 %the upper bound 
 (\ref{15}) we get
\begin{align*}
    \mathbb E \sum_{k=2}^\infty \#\big(\mathcal{N}_{2k}(\mathcal G_{n,\mu}^{(2)})\setminus \mathcal{N}_{2k+1}(\mathcal G_{n,\mu}^{(2)})\big)\le   \sum_{l=2}^{K_\gamma} \mathbb{E}[X_{2l}] + o(sn) \le  \sum_{l=2}^{K_\gamma} \sum_{0<x_{2l}\le m} \phi_{2l} x_{2l}^{-\gamma} + o(sn)
    \\ \le  c m^{1-\gamma}\beta(n)^2 \log\left(\tfrac{1}{\beta(n)}\right) n ^\gamma + o(sn)= c n \beta(n)^{\tfrac{2\gamma}{2\gamma-1}}\log\left(\tfrac{1}{\beta(n)}\right)+o(sn)\,,
\end{align*}
which is $o(sn)$. 
The result now follows from Markov's inequality.
\end{proof}

%\noindent One can see easily, that for two vertices $i,j\le sn$ the probability that they are connected by a 1-connector is strictly greater than that of a direct edge. We do not need to prove this but it gives intuition on what we are about to show. That is, 

We next show, in Lemma~\ref{compofm}, that
almost every vertex, which is in the connected component of some $i\in\{1,\ldots,m\}$, is either its direct neighbour or connected to it by a path using one-connectors, called a \emph{one-path}, which we now define.
%almost every vertex $x>sn$, which is in the connected component of some $i\le m$, is either its direct neighbour or the neighbour of some $j\in\{m+1,\dots,sn\}$ which is connected to~$i$ via a path using only 1-connectors. We define such paths precisely, calling them \emph{1-paths}.}

\begin{definition}\label{1path}\text{}\\[-5mm]
\begin{itemize}
\item    
    \emph{Let $i,j\in\{1,\dots,\lfloor sn \rfloor\}$ and $K_\gamma:=\lfloor\tfrac{1-\gamma}{2\gamma-1}\big\rfloor$. Write $j\overset{1\text{-path}}{\longleftrightarrow}i$} for the event, that there exists an $l\in\{1,\dots, K_\gamma\}$ and distinct vertices $y_1,\dots,y_{l-1}\in\{m+1,\dots,\lfloor sn\rfloor\}$ and $x_1,\dots,x_{l}\in\{\lfloor sn\rfloor+1,\dots,n\}$ such that $i \sim x_1$, $j\sim x_{l}$ and $x_{k} \sim y_k \sim x_{k+1}$ for all $k\in\{1,\dots,l-1\}$. 
\item \emph{If  $i\in\{1,\ldots, \lfloor sn\rfloor\}$, $x\in\{\lfloor sn\rfloor+1,\dots,n\}$ write $x\overset{1\text{-path}}{\longleftrightarrow}i$ for the event that there exists  $j\in\{m+1,\dots,\lfloor sn\rfloor\}$ such that $x\sim j$ and also \smash{$j\overset{1\text{-path}}{\longleftrightarrow}i$.}}
%    {\color{red}Also wenn $x>sn$ haben 1-Pfade ungerade Länge und wenn $x\le sn$ haben sie gerade Länge.}
\end{itemize}
\end{definition}

\noindent Recall that the number of vertices at distances greater than $2K_\gamma+1$ to $\mathcal G^{_{(2)}}_{n,\mu}$ is of a smaller order than $sn$ motivating the restriction $l\leq K_\gamma$ in the definition.
\pagebreak[3]
%{\color{blue}Ich habe die Definition nun wieder nicht asymetrisch gemacht. Ich würde vorschlagen wir schreiben besser sowas wie: "We will use the $\overset{1\text{-path}}{\longleftrightarrow}$ only when considering connections into the set $\{1,\dots,m\}$. Since by Lemma \ref{pathlenght} the number of vertices at distances greater than $2K_\gamma+1$ to $\mathcal G^{_{(2)}}_n$ is of a smaller order than $sn$.}

\begin{lemma}\label{compofm}
Define 
$$\mathcal A_{\text{\emph{no 1-path}}}:=\Bigg\{x\in\{m+1,\dots,n\}\,:\,{ x\leftrightarrow\{1,\dots,m\}, x\not\sim\{1,\dots,m\}\atop \text{ \emph{and} }x\overset{1\text{\emph{-path}}}{\centernot\longleftrightarrow}\{1,\dots,m\}}\Bigg\}\,.$$
Then 
    \begin{align*}
        \frac{1}{n\beta(n)^{\alpha/2}}\#\mathcal A_{\text{\emph{no 1-path}}}
        \longrightarrow \,
        0\qquad \text{ in probability.}
    \end{align*}
\end{lemma}
\begin{proof}
    By Lemma \ref{pathlength} we can neglect vertices at even distance to $\{1,\ldots,m\}$ and at  distances $2\big\lfloor\tfrac{1-\gamma}{2\gamma-1}\big\rfloor+3$
    or greater. 
    For $l\ge 3 $ and $k\leq l-2$ let $X_{l,k}^{\vee},X_{l,k}^{\wedge},X_{l,k}^{\downarrow},X_{l,k}^{\uparrow}$ be the number of paths $x_0x_1\cdots x_l$ such that $x_i>m$ for all $0\leq i<l$, 
    $x_l\le m $ and 
    \begin{itemize}
        \item %$x_k>x_{k+1}$, $x_{k+2}>x_{k+1}$ 
        $x_k, x_{k+2}>x_{k+1}$ 
        and $\big(x_{k+1}>sn$ or $\min\{x_k,x_{k+2}\}<sn\big)$ for $X_{l,k}^{\vee}$,
        \item %$x_k<x_{k+1}$, $x_{k+2}<x_{k+1}$ 
        $x_k, x_{k+2}<x_{k+1}$ 
        and $\big(x_{k+1}\le sn$ or $\max\{x_k,x_{k+2}\}>sn\big)$ for $X_{l,k}^{\wedge}$,
        \item $x_k>x_{k+1}>x_{k+2}$ for $X_{l,k}^{\downarrow}$,
        \item $x_k<x_{k+1}<x_{k+2}$ for $X_{l,k}^{\uparrow}$.
    \end{itemize}
    Also let $Y_l$ be the number of paths $x_0x_1\cdots x_l$ such that $x_0,x_k\in\{m+1,\dots,sn\}$ for all even $k$, $x_k>sn$ for all odd $k\neq l$ and $x_l\le m$ if $l$ is odd and $Y_l=0$ else. That is, the $Y$-paths follow the one-path pattern only up to the last step.\medskip 
    
   We now argue that any path which does not follow the one-path pattern fits at least one of the paths described above. First note that any path ending in $\{1,\ldots,m\}$, which is neither a one-path nor a $Y$-path, has a segment $x_k\sim x_{k+1}$ with either  $x_k,x_{k+1}>sn$ or $x_k,x_{k+1}\in\{m+1,\dots,sn\}$ for some $k\le l-2$. Second, if $x_k>x_{k+1}>sn$ then the next step is either $x_{k+2}<x_{k+1}$ and thus of $X^{\downarrow}$-type or $x_{k+2}>x_{k+1}$ and thus of $X^{\vee}$-type. If however $x_{k+1}>x_k>sn$ then the next step is either $x_{k+2}<x_{k+1}$ and thus of $X^{\wedge}$-type or $x_{k+2}>x_{k+1}$ and thus of $X^{\uparrow}$-type. 
    \smallskip 
    
    A similar argument for the case that $x_k,x_{k+1}\in\{m+1,\dots,sn\}$ shows, that any non one-path fits at least one of the patterns described above. Thus we have that 
    \begin{align}\label{anop}
       \mathbb E\big[ \#\mathcal A_{\text{no 1-path}}\big] \le \sum_{l=3}^{2K_\gamma+1}\sum_{k=1}^{l-2}\mathbb E\big[X_{l,k}^{\vee}+X_{l,k}^{\wedge}+X_{l,k}^{\downarrow}+X_{l,k}^{\uparrow}\big]+\mathbb E\big[Y_l\big]+o(sn)\,.
    \end{align}
    Note that $X_{l,k}^{\vee},X_{l,k}^{\wedge},X_{l,k}^{\downarrow}$ and $X_{l,k}^{\uparrow}$ coincide with $X_l$ from the proof of Lemma \ref{pathlength}, for $\tilde sn=m$, up until step $k$.  
    Thus we can use our recursion (\ref{recurmu}) up to step $k$. We start with $X_{l,k}^{\vee}$, for which we may assume that $x_{k+1}>sn$ and $x_k,x_{k+2}$ arbitrary. Then
    \begin{align*}
       \mathbb E\big[X_{l,k}^{\vee}\big]\le &\sum_{0<x_l\le m} \cdots \sum_{x_k,x_{k+1},x_{k+2}>sn} \mu_k(x_k)\beta(n)^2x_k^{\gamma-1}x_{k+1}^{-2\gamma}x_{k+2}^{\gamma-1}\prod_{i=k+2}^{l-1} \mathbb P (x_i\sim x_{i+1})
    \end{align*}
    Using Lemma \ref{1.3} and estimating the sum by an integral we get for the middle term
    \begin{align*}
        \sum_{x_k,x_{k+1}>sn} & \mu_k(x_k)\beta(n)^2x_k^{\gamma-1}x_{k+1}^{-2\gamma}x_{k+2}^{\gamma-1}
        \\ & \le  \sum_{x_k,x_{k+1}>sn} \big(\mathbbm{1}_{\{x_k>sn\}}\psi_kx_k^{2(\gamma-1)}+\phi_kx_k^{-1}\big)\beta(n)^2x_{k+1}^{-2\gamma}x_{k+2}^{\gamma-1}
        \\ & \le c\mathbbm{1}_{\{x_{k+2}>sn\}}\big( \psi_k \beta(n)+\phi_k \beta(n)^2\log(\tfrac1s)(sn)^{1-2\gamma}\big)x_{k+2}^{\gamma-1}
         \\ & \le    c'\mathbbm{1}_{\{x_{k+2}>sn\}} \psi_{k} \beta(n) x_{k+2}^{\gamma-1},
    \end{align*}
    where we used (\ref{psiphi}) and (\ref{15}) in the last step with $\tilde s=\beta(n)^\alpha$. By the same set of equations one can also see that $\beta(n)\psi_k=o(\psi_{k+1})$.
    For $X_{l,k}^{\wedge}$ we may assume without loss of generality that $x_k>sn$, then
    \begin{align*}
        \mathbb E\big[X_{l,k}^{\wedge}\big]\le &\sum_{0<x_l\le m} \cdots \sum_{x_k>sn\atop x_{k+1},x_{k+2}>m} \mu_k(x_k)\beta(n)^2x_k^{-\gamma}x_{k+1}^{2(\gamma-1)}x_{k+2}^{-\gamma}\prod_{i=k+2}^{l-1} \mathbb P (x_i\sim x_{i+1})
    \end{align*}
    and for the middle term
    \begin{align*}
        \sum_{x_k>sn\atop x_{k+1}>m} \mu_k(x_k)\beta(n)^2x_k^{-\gamma}x_{k+1}^{2(\gamma-1)}x_{k+2}^{-\gamma}
    \le &\,c\big(\psi_k \log(\tfrac1s)+\phi_k(sn)^{1-2\gamma}\big)\beta(n)^2n^{2\gamma-1}x_{k+2}^{-\gamma}
         \\ \le &\,c'\phi_{k+1}\beta(n)\log(\tfrac1s)x_{k+2}^{-\gamma}\,.
    \end{align*}
    For the other paths we have
    \begin{align*}
         \mathbb E\big[X_{l,k}^{\downarrow}\big]\le &\sum_{0<x_l\le m} \cdots \sum_{x_k,x_{k+1},x_{k+2}>m} \mu_k(x_k)\beta(n)^2x_k^{\gamma-1}x_{k+1}^{-1}x_{k+2}^{-\gamma}\prod_{i=k+2}^{l-1} \mathbb P (x_i\sim x_{i+1})\,.
    \end{align*}
    Again, by our estimates (\ref{15}) we have for the middle term 
    \begin{align*}
         \sum_{x_k,x_{k+1}>m} & \mu_k(x_k)\beta(n)^2x_k^{\gamma-1}x_{k+1}^{-1}x_{k+2}^{-\gamma}
        \\\le & \sum_{x_k,x_{k+1}>m} \big(\mathbbm{1}_{\{x_k>m\}}\psi_kx_k^{2(\gamma-1)}+\phi_kx_k^{-1}\big)\beta(n)^2x_{k+1}^{-1}x_{k+2}^{-\gamma}
        \\ \le & c\big( \psi_k n^{2\gamma-1}\log(\tfrac1s)\beta(n)^2+\phi_k \beta(n)^2\log(\tfrac1s)^2\big)x_{k+2}^{-\gamma}
         \\ \le & c'\beta(n)\log(\tfrac1s) \phi_{k+1}  x_{k+2}^{-\gamma} \;.
    \end{align*}
    Similarly one can see, that 
     \begin{align*}
         \mathbb E\big[X_{l,k}^{\uparrow}\big]\le &\sum_{0<x_l\le m} \cdots \sum_{x_{k+2}>m} c'\mathbbm{1}_{\{x_{k+2}>m\}}\beta(n)\log(\tfrac1s)\psi_{k+1}x_{k+2}^{\gamma-1}\prod_{i=k+2}^{l-1} \mathbb P (x_i\sim x_{i+1})\,.
    \end{align*}
    Thus in any of the four paths we can factor out the $c'\beta(n)\log(\tfrac1s)$ term and continue the recursion. Hence
    \begin{align*}
         \mathbb E\big[X_{l,k}^{\vee}+X_{l,k}^{\wedge}+X_{l,k}^{\uparrow}+X_{l,k}^{\downarrow}\big]\le c'\beta(n)\log(\tfrac1s) \mathbb E\big[X_{l}\big]\,,
    \end{align*}
    for any $k$. Moreover for the paths of uneven length, we get by choosing $\tilde s n=m=\beta(n)^{2/(2\gamma-1)}n$ in (\ref{15}),   that
    $$\phi_{2l+1}\le c \beta(n)^{2l+1}s^{l(2\gamma-1)}n^\gamma = c \beta(n) n^\gamma \,.$$
    Thus by (\ref{14}) and the above we get
    \begin{align*}
           \sum_{l=1}^{K_\gamma} \beta(n)\log(\tfrac1s) \mathbb E\big[X_{2l+1}\big] + o(sn)
          \le &  \sum_{l=1}^{K_\gamma} C\beta(n)\log(\tfrac1s)\phi_{2l+1}m^{1-\gamma} + o(sn) \\
          \le & C\beta(n)^{2+(1-\gamma)\alpha}\log(\tfrac1s) n +o(sn)=o(sn)\,.
    \end{align*}
    It remains to estimate the sum over the $Y$-paths in (\ref{anop}). By the definition of $Y_l$ we have, for $l\ge 3$ odd, that 
    \begin{align*}
          \mathbb E\big[Y_l\big]\le &\sum_{x_l\le m}\sum_{m<x_0,x_2,..,x_{l-1}\le sn }\prod_{i=0}^{l-3}\mathbb P(x_i\overset{1}{\longleftrightarrow}x_{i+2})\beta(n)x_{l-1}^{\gamma-1}x_l^{-\gamma}
          \\\le &\sum_{x_l\le m}\sum_{m<x_0,x_{l-1}\le sn}\frac{\beta(n)^3n^{2\gamma-1}}{2\gamma-1}x_0^{-\gamma}x_{l-1}^{-1}x_l^{-\gamma}\Big(\sum_{m<x\le sn}\frac{\beta(n)^2n^{2\gamma-1}}{2\gamma-1}x^{-2\gamma}\Big)^{\tfrac{l-3}{2}}
           \\\le &\,c_1\beta(n)^3n^{2\gamma-1}(sn)^{1-\gamma}\log(\tfrac1s)m^{1-\gamma} c^l_2 = \,c_1 c_2^l \beta(n)^{\tfrac{3\gamma}{2\gamma-1}}n =o(sn)\,.
    \end{align*}
    Since we only sum over finitely many $l$ we are done.
\end{proof}

At this point we know that the size of the component of $C_1(\mathscr G_{n,\mu}^{_{(2)}})$ is asymptotically the same as the number of vertices connected to it either by an edge or by a one-path. We continue by systematically ruling out further sets of vertices, first those
with index smaller than $sn$. %\beta(n)^{\alpha/2}$.}

\begin{lemma} \label{3.12}
We have
    \begin{align*}
        \frac{1}{n\beta(n)^{\alpha/2}}\#
        \big\{ j \in \{1,\ldots, \lfloor sn
        %n\beta(n)^{\alpha/2} 
        \rfloor \}  \colon j \sim \mathcal G^{(2)}_{n,\mu} \text{ or }
        j \overset{1\text{\emph{-path}}}{\longleftrightarrow} \mathcal G^{(2)}_{n,\mu} \big\}
        {\longrightarrow} \, 0\quad \text{ in probability.}
    \end{align*}
\end{lemma}

\begin{proof} Note first that we can restrict attention to
    vertices $j\in \{m+1,\dots, \lfloor sn \rfloor \}$. If $j$ is  connected by a one-path to some $i \in \mathcal G^{_{(2)}}_{n,\mu}$ then it is by definition at an even distance to $i$ and by Lemma~\ref{pathlength} this set is negligible. Moreover, we have
    \begin{align*}
   \mathbb E \big[\#\{j\in\{m+1,\dots,\lfloor sn \rfloor\}\,:\,j\sim \mathcal G^{(2)}_{n, \mu} \big]\le &\sum_{i\le m \atop j\in\{m+1,\dots,\lfloor sn \rfloor\}} \beta(n) i^{-\gamma}j^{\gamma-1}   \\ \le&\, C\beta(n)s^{2-\gamma}n =o(sn),
   \end{align*}
   showing that nearest neighbours in this set are also negligible.
\end{proof}

Our aim is now to show that, for every $\varepsilon, \delta>0$, an $M$ can be chosen so large that with probability exceeding $1-\delta$ all but $\varepsilon n\beta(n)^{\alpha/2}$ vertices, which are connected by a path to the giant in $\mathscr G_{n,\mu}^{_{(2)}}$, are connected by an edge to a unique vertex in $\mathscr G_{n,M}^{_{(2)}}$, see Lemma~\ref{uppersuper}. We prepare this step with a further lemma, comparing connections by one-paths and by one-connectors.

\begin{lemma}\label{1con}
    There exist $c, c'>1$, depending on $\mu$, such that, for any $i,j\in\{1,\dots,\lfloor sn\rfloor\}$ and $x\in\{\lfloor sn\rfloor+1,\dots,n\}$, we have
    \begin{align*}
        \mathbb P(j\overset{1\text{\emph{-path}}}{\longleftrightarrow}i) \le c \,\mathbb P(j\overset{1}{\longleftrightarrow}i)\quad\text{ and }\quad\mathbb P(x\overset{1\text{\emph{-path}}}{\longleftrightarrow}i) \le c' \,\mathbb P(x\sim i)\,.
    \end{align*}
    Here %, by slight abuse of notation, 
    $j\overset{1}{\longleftrightarrow}i$ stands for the event that $i$ and $j$ are connected via a single middle vertex in $\{\lfloor sn\rfloor+1,\dots,n\}$, in accordance with Definition~\ref{1path}. 
\end{lemma}
\begin{proof}
From (\ref{rankone}) we get
\begin{equation}\label{simpleform}
\begin{aligned}
    \mathbb P(j\overset{1}{\longleftrightarrow}i) & = 1- \exp\Big(- \frac{1}{m}\, \kappa_\mu\Big( \frac{i}{m}, \frac{j}{m}\Big)\big(1+o(1)\big)\Big) \\ &  = \tfrac{\beta(n)^2n^{2\gamma-1}}{2\gamma-1} i^{-\gamma} j^{-\gamma} (1+o(1)),
    \end{aligned}
    \end{equation}
     using that $i \vee j \ge m$ and hence
    $\beta(n)^2n^{2\gamma-1}(i \vee j)^{-\gamma} \le \beta(n)^{2-\gamma\alpha}n^{\gamma-1} 
    \to 0\,.$
By the definition of a one-path the maximum path length is bounded by some~$C=2K_\gamma+1$. Therefore, by Markov’s inequality,
\begin{align*}
     \mathbb P(j\overset{1\text{-path}}{\longleftrightarrow}i) & \le  \sum_{k=0}^C \sum_{m<y_1,..,y_{k}\le sn}\mathbb P(i\overset{1}{\longleftrightarrow}y_1)\mathbb P(y_{k}\overset{1}{\longleftrightarrow}j)\prod_{i=1}^{k-1}\mathbb P(y_i\overset{1}{\longleftrightarrow}y_{i+1})
     \\ & \le   i^{-\gamma} j^{-\gamma}\sum_{k=0}^C \sum_{m<y_1,..,y_{k}\le sn}\Big( \frac{\beta(n)^2n^{2\gamma-1}}{2\gamma-1}\Big)^{k+1} y_1^{-2\gamma}\cdots y_k^{-2\gamma}
     \\ 
     & \le   i^{-\gamma} j^{-\gamma} \frac{\beta(n)^2n^{2\gamma-1}}{2\gamma-1} \sum_{k=0}^C
     \Big( \frac{\beta(n)^2n^{2\gamma-1}}{2\gamma-1} \sum_{y>m} y^{-2\gamma} \Big)^{k}   \\
     & \le    c \, i^{-\gamma} j^{-\gamma} \frac{\beta(n)^2n^{2\gamma-1}}{2\gamma-1} \,,
     %\sum_{k=0}^C \Big( \frac{1}{2\gamma-1}\Big)^{2k}
\end{align*}
for some $c>1$ depending on $\mu$. 
\pagebreak[3]\smallskip

Now let $x>sn$, then by Definition \ref{1path} and the above we have
\begin{align*}
     \mathbb P(x\overset{1\text{-path}}{\longleftrightarrow}i)= &\,\mathbb P(\exists j\in\{m+1,\dots,sn\}\,:\,x\sim j \text{ and }j\overset{1\text{-path}}{\longleftrightarrow}i)
     \\ \le &\, \sum_{m< j \le sn}\beta(n)x^{\gamma-1}j^{-\gamma} c \mathbb P(j\overset{1}{\longleftrightarrow}i)
     \\ \le &\, c' \beta(n)^3 m^{1-2\gamma}n^{2\gamma-1}x^{\gamma-1}i^{-\gamma}=c' \beta(n) x^{\gamma-1}i^{-\gamma}\,,
\end{align*}
for a suitable $c'>1$  depending on $\mu$, which completes the proof.
\end{proof}

\pagebreak[3]
%\noindent Next we show that the vertices in $\{sn+1,\dots,n\}$, with the largest contribution to $C_1$, have a unique  parent of order $m$, however this parent is not necessarily in $\{1,\dots,m\}$.
%{\color{red}Every vertex contributes one to the size of the component if it is in the component and zero otherwise. So what exactly do you count if you say the vertex with the largest contribution?  Also: Is $C_1$ the largest component in $\mathscr G_n$ or the component of $\mathcal G_n^{_{(2)}}$ but then for which $\mu$? We have not shown that this is the same. How do you define parent of a vertex? I would have thought it is defined as any element in $\{1,\ldots,m\}$ which is at minimal distance to the vertex, but this you say it isn't. }
%{\color{blue}The largest component in $\mathcal G_{n,\mu}^{_{(2)}}$ induces a component in $\mathscr G_n$ with size of order $sn$. What we want to show is that, for some large $M>1$,  $$\#\{x>sn\,:\,x\leftrightarrow\{1,\dots,m\}\}\approx \#\Big\{x>sn\,:\,{\exists j \le mM  \,:j\sim x \atop j\overset{1\text{-path}}{\longleftrightarrow} \{1,\dots,m\}}\Big\}\,.$$ In fact we can make $M$ large enough, such that the  ratio of the left hand side and the right hand side becomes arbitrarily close to one.}

\begin{lemma}\label{orderm}
%Let $O_M:=\{m+1,\dots, \lceil M \beta(n)^{\alpha}n\rceil \}$ for some  fixed $M> \mu$. 
%In other words $O_M$ is the set of all vertices of order $m$, which are not included in $\{1,\dots,m\}$. Then
\ \\[-5mm]
\begin{itemize}
    \item[(a)] $\displaystyle
   \frac{1}{sn}  \#\big\{x>sn\,:\,\exists i_1,i_2\le sn \text{ \emph{with} }i_1\sim x\sim i_2\big\} \longrightarrow \,0$ in probability.
\item[(b)] For every $\mu>0$ there exists $C=C(\mu)>0$ such that, for any $\varepsilon>0$ and $M>\mu$ we have,
   \begin{align*}
    \mathbb P \Big(
      \#\big\{x>sn\,:\, & x\overset{1\text{\emph{-path}}}{\longleftrightarrow}\{1,\dots,m\} \text{\emph{ and }} \\  & x\not\sim \{1,\dots, \lfloor M \beta(n)^{\alpha}n\rfloor\}\big\} > \varepsilon sn\Big) \le C  \varepsilon^{-1} M^{1-2\gamma}\,.
\end{align*}
\end{itemize}
\end{lemma}

\begin{proof}
    We have 
    \begin{align*}
        \mathbb E\big[\#\{x>sn\,:\,\exists i_1,i_2\le sn \text{ with }i_1\sim x\sim i_2\}\big]\le & \sum_{x>sn}\sum_{i_1,i_2\le sn}\beta(n)^2 i_1^{-\gamma}i_2^{-\gamma}x^{2(\gamma-1)}
        \\ \le & C \beta(n)^{2+\tfrac{2(1-\gamma)}{2\gamma-1}}n =o(sn)\,.
    \end{align*}
    By Markov's inequality, this already yields statement $(a)$ of the lemma.\smallskip \\
    \noindent    For the set of vertices in $(b)$ we have by definition
    \begin{align*}
        \#\big\{x>sn\,:\, & x\overset{1\text{-path}}{\longleftrightarrow}\{1,\dots,m\} \text{ and } x\not\sim \{1,\dots, \lfloor M \beta(n)^{\alpha}n\rfloor\}\big\}
       \\ &  \le  \#\bigg\{x>sn\,:\,{\exists y\in\{\lfloor M \beta(n)^{\alpha}n\rfloor+1,\dots,\lfloor sn \rfloor\} \atop x\sim y, \;y\overset{_{1\text{-path}}}{\longleftrightarrow}\{1,\dots,m\}} \bigg\} \,.
    \end{align*}
    Taking expectations, for fixed $x$ the event \smash{$\{ \exists y\in\{\lfloor M \beta(n)^{\alpha}n\rfloor+1,\dots,\lfloor sn \rfloor \} \colon x\sim y,$} \smash{$y\overset{_{1\text{-path}}}{\longleftrightarrow}\{1,\dots,m\}\}$}
is contained in the union of the  events that there is a one-path from
    $\{\lfloor M \beta(n)^{\alpha}n\rfloor+1,\dots,\lfloor sn \rfloor \}$ to $\{1,\dots,m\}$ using the vertex $x$ and the event that there is such a one-path not using the vertex $x$. If a one-path uses vertex~$x$, then $x$ is a one-connector, the number of which is $o(sn)$ in probability by $(a)$. Otherwise, we use a union bound in $y$ and observe that for fixed $y$ the events $\{x\sim y \}$ and $\{y\overset{_{1\text{-path}}}{\longleftrightarrow}\{1,\dots,m\}$ not over $x\}$ are independent.
    Thus, by Lemma~\ref{1con}, the expectation of the right hand side can be bounded from above by 
    \begin{align*}
        \sum_{x>sn}  & \sum_{y=\lfloor M \beta(n)^{\alpha}n\rfloor+1}^{\lfloor sn \rfloor}\sum_{z\le m} \mathbb P(x\sim y)\mathbb P(y\overset{_{1\text{-path}}}{\longleftrightarrow}z) + o(sn)
      \\ & \le  c\sum_{x>sn} \sum_{y=\lfloor M \beta(n)^{\alpha}n\rfloor+1}^{\lfloor sn \rfloor}\sum_{z\le m}\beta(n)^3n^{2\gamma-1}x^{\gamma-1}y^{-2\gamma}z^{-\gamma} +o(sn)
      \\ & \le C\beta(n)^{\tfrac{4\gamma-1}{2\gamma-1}}n^{2\gamma}\big(M  \beta(n)^{\tfrac{2}{2\gamma-1}}n\big)^{1-2\gamma}+o(sn)\,,
    \end{align*}
    for some $C>0$ depending only on  $\mu$ and~$\gamma$.
    As
    $\beta(n)^{\tfrac{4\gamma-1}{2\gamma-1}}n^{2\gamma} \beta(n)^{-2}n^{1-2\gamma}=sn$,
    the statement~$(b)$ follows by Markov's inequality.
\end{proof}
%\pmar{{\color{blue} \noindent Die "Konstante" verhält sich wie $C\asymp \mu^{1-\gamma}$. Daher brauchen wir nur $M>> \mu^{\frac{1-\gamma}{2\gamma-1}}$.}{\color{red} Ich dachte $\mu$ sei jetzt festgehalten?}{\color{blue}genau, bei festgehaltenem $\mu$ kannst du $M$ so groß machen wie es das Lemma behauptet ( und es muss mindestens so groß sein wie $\mu^{\frac{1-\gamma}{2\gamma-1}}$ .} }

Let $K$ be a component of $\mathcal G^{(2)}_{n,\mu}$ and denote by $\mathcal C(K)$ its component in $\mathscr G_n$. 
%The following lemma enables us to bound the size of the component of the giant in $\mathcal G^{(2)}_{n,\mu}$.

\begin{lemma}\label{uppersuper}
    For any $\varepsilon,\delta\in(0,1)$ there exists $M_0$  such that, for all $M\ge M_0$,
    $$\liminf_{n\to\infty}\mathbb P\Big(\#\mathcal C(C_1(\mathcal G^{(2)}_{n,\mu})) \le  \sum_{i\in C_1(\mathcal G^{(2)}_{n,M})} {\mathrm{outdeg}}(i) \,+\varepsilon sn\Big) \ge 1-\delta,$$
    and 
    $$\lim_{n\to\infty} \frac{1}{sn}\sum_{i\in C_1(\mathcal G^{(2)}_{n,M})} {\mathrm{outdeg}}(i)  =
     M^{1-\gamma} \int_0^1 f(x) \rho(\kappa_M , x)\,\text{d}x\,,$$
     in probability.
\end{lemma}

\begin{proof}
Let $O_M:=\{m+1,\dots, \lfloor M \beta(n)^{\alpha}n\rfloor \}$, roughly speaking $O_M$ is the set of all vertices of order $m$, which are not included in $\{1,\dots,m\}$. 
By Lemma \ref{3.12} and \ref{orderm}, for any $\varepsilon,\delta\in(0,1)$ we can choose $M$ large enough, such that, for all sufficiently large $n$,
\begin{align}\label{bipartle}
    \#\mathcal C(C_1(\mathcal G^{(2)}_{n,\mu})) \le  \sum_{i\in C_1(\mathcal G^{(2)}_{n,\mu})} \text{outdeg}(i) \,+\sum_{j\in O_M \atop j\overset{1\text{-path}}{\longleftrightarrow}C_1(\mathcal G^{(2)}_{n,\mu})}\text{outdeg}(j)+\varepsilon sn
\end{align}
holds with probability at least $1-\delta$.
%Also \begin{align}\label{bipartge}
%    \#\mathcal C(C_1(\mathcal G^{(2)}_n(m))) \ge  \sum_{i\in C_1(\mathcal G^{(2)}_n(m))} \text{outdeg}(i) +\sum_{j\in O_M \atop j\overset{1\text{-path}}{\longleftrightarrow}C_1(\mathcal G^{(2)}_n(m))}\text{outdeg}(j)+o( sn)
%\end{align}
%with high probability {\color{cyan}as $n\to\infty$.} 
%{\color{red}w.h.p. ist bereits als Grenzwertaussage definiert, "as $n\to\infty$" ist also redundant.} {\color{blue}Ich möchte damit sagen, dass "with high probability as $n\to\infty$" gemeint ist und nicht "with high probability as $\mu\to\infty$" oder sowas. Das ist zwar redundant aber doch hilfreich.}
%\pagebreak[3]
\smallskip

%\noindent 
%Recall that $m=%\lfloor \mu \beta(n)^{2/(2\gamma-1)}n\rfloor = 
%\lfloor\mu \beta(n)^{\alpha}n\rfloor$. 
Recall the definition of one-paths from Definition~\ref{1path} and
in particular the implicitly defined even steps $y_1,\ldots, y_{l-1}
\in\{m+1,\ldots, \lfloor sn \rfloor\}$. Let
$$O_M':=\{ j\in O_M \colon j\overset{1\text{-path}}{\longleftrightarrow}C_1(\mathcal G^{(2)}_{n,\mu}) \text{ with }
y_1,\ldots, y_{l-1} \in O_M\}$$
be the vertices in $O_M$ connected to $\mathcal G^{(2)}_{n,\mu}$
by one-paths using only vertices in $O_M$ at the even steps.
If for all $\varepsilon,\delta \in (0,1)$ we can make $M$ large enough, such that
\begin{align}\label{3.17}
    \mathbb P\bigg(\sum_{{j\in O_M\setminus O_M' \colon}\atop{ j\overset{1\text{-path}}{\longleftrightarrow}C_1(\mathcal G^{(2)}_{n,\mu})}} \text{outdeg}(j) > \varepsilon sn\bigg) < \delta,
\end{align}
then the second sum in (\ref{bipartle}) can be reduced to $j\in O_M'$ with 
$j\overset{1\text{-path}}{\longleftrightarrow}C_1(\mathcal G^{_{(2)}}_{n,\mu})$
and these $j$ are  in the graph
$\mathcal G^{_{(2)}}_{n,M}$ connected to 
$C_1(\mathcal G^{_{(2)}}_{n,\mu})$. As $C_1(\mathcal G^{_{(2)}}_{n,\mu})\subset C_1(\mathcal G^{_{(2)}}_{n,M})$ by the uniqueness of the giant in supercritical inhomogeneous random graphs,
we can replace the sums in (\ref{bipartle}) by a single sum over vertices $i\in C_1(\mathcal G^{_{(2)}}_{n,M})$. \pagebreak[3]
\smallskip

 It remains to show \eqref{3.17}. Let $\varepsilon,\delta \in (0,1)$. For $j\in O_M\setminus O_M'$ with $j\overset{_{1\text{-path}}}{\longleftrightarrow}C_1(\mathcal G^{_{(2)}}_{n,\mu})$. 
 %We may assume that the shortest path from $j$ to $C_1(\mathcal G^{(2)}_{n,\mu})$ has length greater than two and uses only vertices with index greater than $M\beta(n)^\alpha n$. \pmar{So nicht richtig.} 
 there exists \smash{$y\in\{\lfloor M \beta(n)^{\alpha}n\rfloor+1,\dots,\lfloor sn\rfloor\}$} such that 
 $$\{j\overset{1\text{-path}}{\longleftrightarrow}y\} \circ \{y\overset{{1\text{-path}}}{\longleftrightarrow}C_1(\mathcal G^{(2)}_{n, \mu}) \},$$ where $\circ$ denotes the disjoint occurrence 
 of events. It therefore suffices to show that, for large enough $M$, 
\begin{align}\label{3.18}
    \sum_{j\in O_M } \sum_{y=\lfloor M \beta(n)^{\alpha}n\rfloor+1}^{\lfloor sn\rfloor} \mathbbm{1}_{\{j\overset{_{1\text{-path}}}{\longleftrightarrow}y \} \,\circ \,\{y\overset{_{1\text{-path}}}{\longleftrightarrow}C_1(\mathcal G^{(2)}_{n,\mu})\}}\text{outdeg}(j)\le \varepsilon s n
\end{align}
with probability $1-\delta$, for all sufficiently large $n$.
\pagebreak[3]\smallskip

Recall that $\text{outdeg}(j)=\sum_{x>j}\mathbbm{1}_{\{x\sim j\}}$. If 
$x \le sn $ the edge $\{x,j\}$ cannot be used in a one-path and so the events $\{x\sim j\}$ and \smash{$\{j\overset{_{1\text{-path}}}{\longleftrightarrow}y\}$}, as well as $\{x\sim j\}$ and \smash{$\{y\overset{{1\text{-path}}}
{\longleftrightarrow}C_1(\mathcal G^{(2)}_{n, \mu}) \}$}  are independent. 
If $x>sn$ we have 
\begin{align*}
    \mathbb P & \big( (\{j\overset{_{1\text{-path}}}{\longleftrightarrow}y\}  \,\circ \, \{y\overset{_{1\text{-path}}}{\longleftrightarrow}C_1(\mathcal G^{(2)}_{n,\mu})\}) \cap \{ x \sim j\}  \big)
    \\ & = \mathbb P \big(x \sim j \big) \mathbb P \big( \{j\overset{_{1\text{-path}}}{\longleftrightarrow}y \} \,\circ \, \{y\overset{_{1\text{-path}}}{\longleftrightarrow}C_1(\mathcal G^{(2)}_{n,\mu}) \} \,\big|\,x \sim j \big),
\end{align*}  
and we can decompose the event in the conditional probability
\begin{align}
     \{j  \overset{_{1\text{-path}}}{\longleftrightarrow} & y \} \,\circ \, \{y\overset{_{1\text{-path}}}{\longleftrightarrow}C_1(\mathcal G^{(2)}_{n,\mu}) \} \notag
   \\  \subseteq & \big( \{j\overset{_{1\text{-path}}}{\longleftrightarrow}y \text{ not using } x\sim j \} \,\circ \, \{y\overset{_{1\text{-path}}}{\longleftrightarrow}C_1(\mathcal G^{(2)}_{n,\mu})\text{ not using } x\sim j \} \big) \label{(3.20)}
   \\ & \cup \big( \{j\overset{_{1\text{-path}}}{\longleftrightarrow}y \text{ using } x\sim j \} \,\circ \, \{y\overset{_{1\text{-path}}}{\longleftrightarrow}C_1(\mathcal G^{(2)}_{n,\mu})\text{ not using } x\sim j \} \big) \label{(3.21)}
    \\ & \cup \big( \{j\overset{_{1\text{-path}}}{\longleftrightarrow}y \text{ not using } x\sim j \} \,\circ \, \{y\overset{_{1\text{-path}}}{\longleftrightarrow}C_1(\mathcal G^{(2)}_{n,\mu})\text{ using } x\sim j \} \big),\label{(3.22)}
\end{align}
as not both paths can use the edge $\{x, j\}$ for the disjoint occurrence. Conditioning on the event $x\sim j$, the event in line \eqref{(3.20)} is unaffected, the first event in line \eqref{(3.21)} becomes
\smash{$\{x\overset{_{1\text{-path}}}{\longleftrightarrow}y$ not using $x\sim j$ $\}$} and the second event in line \eqref{(3.22)}  
%implies the existence of a one-path from $y$ to $j$ using $x\sim j$ 
implies the existence of a one-path from $y$ to $x$ or from  $y$ to $j$ not using $x\sim j$. Hence, 
\begin{align*}    
     \mathbb P \big( & \{j\overset{_{1\text{-path}}}{\longleftrightarrow}y \} \,\circ \, \{y\overset{_{1\text{-path}}}{\longleftrightarrow}C_1(\mathcal G^{(2)}_{n,\mu}) \} \,\big|\,x \sim j \big)
    \\ & \leq \,  \mathbb P \big(\{j\overset{_{1\text{-path}}}{\longleftrightarrow}y\text{ not using } x\sim j \} \circ \{y\overset{_{1\text{-path}}}{\longleftrightarrow}C_1(\mathcal G^{(2)}_{n,\mu})\text{ not using } x\sim j \} \big)
    \\ & \qquad +\mathbb P \big( \{x \overset{_{1\text{-path}}}{\longleftrightarrow} y \text{ not using } x\sim j \} \,\circ \, \{y\overset{_{1\text{-path}}}{\longleftrightarrow}C_1(\mathcal G^{(2)}_{n,\mu})\text{ not using } x\sim j  \}\big) 
    \\ & \qquad + \mathbb P \big(\{j\overset{_{1\text{-path}}}{\longleftrightarrow}y \text{ not using } x\sim j \} \circ \{y \overset{_{1\text{-path}}}{\longleftrightarrow} x \text{ not using } x\sim j \} \big)
    \\ & \qquad + \mathbb P \big(\{j\overset{_{1\text{-path}}}{\longleftrightarrow}y \text{ not using } x\sim j \} \circ \{j\overset{_{1\text{-path}}}{\longleftrightarrow}y \text{ not using } x\sim j \} \big) 
    \\ & \le  \Big( \mathbb P  \big(j\overset{_{1\text{-path}}}{\longleftrightarrow}y \text{ not using } x\sim j  \big)  +  \mathbb P \big( y \overset{_{1\text{-path}}}{\longleftrightarrow} x \big) \Big)\mathbb P \big( y\overset{_{1\text{-path}}}{\longleftrightarrow}C_1(\mathcal G^{(2)}_{n,\mu})\big) 
    \\ & \qquad + \mathbb P \big(j\overset{_{1\text{-path}}}{\longleftrightarrow}y \text{ not using } x\sim j \big)\Big(\mathbb P \big( y \overset{_{1\text{-path}}}{\longleftrightarrow} x  \big)+\mathbb P \big(j\overset{_{1\text{-path}}}{\longleftrightarrow}y \text{ not using } x\sim j \big)\Big),
\end{align*}
where we used the BK-inequality in the last step.
% Continuing  
%\begin{align*}    
%    & \mathbb P \big(x \sim j \big) \mathbb P \big( \{j\overset{_{1\text{-path}}}{\longleftrightarrow}y \} \,\circ \, \{y\overset{_{1\text{-path}}}{\longleftrightarrow}C_1(\mathcal G^{(2)}_{n,\mu}) \} \,\big|\,x \sim j \big)
%    \\ & \leq \, \mathbb P \big( x \sim j \big) \mathbb P \big(\{j\overset{_{1\text{-path}}}{\longleftrightarrow}y\text{ not using } x\sim j \} \circ \{y\overset{_{1\text{-path}}}{\longleftrightarrow}C_1(\mathcal G^{(2)}_{n,\mu})\} \big)
%    \\ & \qquad +\mathbb P \big( x \sim j \big) \mathbb P \big( \{y \overset{_{1\text{-path}}}{\longleftrightarrow} x \} \,\circ \, \{y\overset{_{1\text{-path}}}{\longleftrightarrow}C_1(\mathcal G^{(2)}_{n,\mu}) \}\big) 
%    \\ & \qquad {\color{blue}+ \mathbb P \big( x \sim j \big)\mathbb P \big(\{j\overset{_{1\text{-path}}}{\longleftrightarrow}y \text{ not using } x\sim j \} \circ \{y \overset{_{1\text{-path}}}{\longleftrightarrow} x \} \big) }
%    \\ & \le \mathbb P \big( x \sim j \big) \Big( \mathbb P  \big(j\overset{_{1\text{-path}}}{\longleftrightarrow}y \text{ not using } x\sim j  \big)  +  \mathbb P \big( y \overset{_{1\text{-path}}}{\longleftrightarrow} x \big) \Big)\mathbb P \big( y\overset{_{1\text{-path}}}{\longleftrightarrow}C_1(\mathcal G^{(2)}_{n,\mu})\big) 
%    \\ & \qquad {\color{blue}+ \mathbb P \big( x \sim j \big)\mathbb P \big(j\overset{_{1\text{-path}}}{\longleftrightarrow}y \text{ not using } x\sim j \big)\mathbb P \big( y \overset{_{1\text{-path}}}{\longleftrightarrow} x  \big) },
%\end{align*} where we used the BK-inequality in the last step.\smallskip
By Lemma~\ref{1con} and the above there exists  $c>0$ such that the expectation of the left-hand side in (\ref{3.18}) is bounded from above by 
\begin{align*}
    \sum_{j\in O_M} & \sum_{y=\lfloor M \beta(n)^{\alpha}n\rfloor}^{\lfloor sn\rfloor}\sum_{x=j+1}^n \sum_{i=1}^m c  \mathbb P \big( x \sim j \big) \Big(\mathbb P(j\overset{_{1\text{-path}}}{\longleftrightarrow}y) +  \mathbbm{1}_{\{x>sn\}}\mathbb P \big( y \sim x \big)\Big)\mathbb P(y\overset{_{1\text{-path}}}{\longleftrightarrow}i)
    \\ & \qquad+  \sum_{j\in O_M}  \sum_{y=\lfloor M \beta(n)^{\alpha}n\rfloor}^{\lfloor sn\rfloor}\sum_{x=j+1}^n\mathbb P \big( x \sim j \big)\mathbb P \big(j\overset{_{1\text{-path}}}{\longleftrightarrow}y \big)\Big(\mathbb P \big( y \sim  x  \big)+ \mathbb P \big(j\overset{_{1\text{-path}}}{\longleftrightarrow}y \big)\Big)
    \\ & \le \sum_{j\in O_M} \sum_{y=\lfloor M \beta(n)^{\alpha}n\rfloor}^{\lfloor sn\rfloor}\sum_{i=1}^m c \beta(n)^4 n^{4\gamma-2}j^{-\gamma}y^{-2\gamma}i^{-\gamma}\Big(\beta(n)n^\gamma j^{-\gamma}+1 \Big)
      \\ & \qquad +  \sum_{j\in O_M}  \sum_{y=\lfloor M \beta(n)^{\alpha}n\rfloor}^{\lfloor sn\rfloor}\sum_{x=j+1}^n c \beta(n)^4 n^{2\gamma-1} x^{\gamma-1} j^{-2\gamma} y^{-2\gamma}\Big( x^{\gamma-1} + \beta(n)n^{2\gamma-1}j^{-\gamma}\Big)
      \end{align*}
      \begin{align*}
    \\ &\le   c \beta(n)^{5}n^{5\gamma-2}
    m^{2-3\gamma} 
    \big(M  \beta(n)^{\tfrac{2}{2\gamma-1}}n\big)^{1-2\gamma} + c \beta(n)^4 n^{4\gamma-2} m^{1-\gamma } \big(M  \beta(n)^{\tfrac{2}{2\gamma-1}}n\big)^{2-3\gamma} 
    \\ & \qquad + c\beta(n)^4 n^{4\gamma-2}m^{1-2\gamma} \big(M  \beta(n)^{\tfrac{2}{2\gamma-1}}n\big)^{1-2\gamma} + c \beta(n)^{5}n^{5\gamma-2}
    m^{1-3\gamma} 
    \big(M  \beta(n)^{\tfrac{2}{2\gamma-1}}n\big)^{1-2\gamma}
    \\  & \le c\beta(n)^{\tfrac{4\gamma-1}{2\gamma-1}}n^{2\gamma}\big(M  \beta(n)^{\tfrac{2}{2\gamma-1}}n\big)^{1-2\gamma} + c \beta(n)^{\tfrac{6\gamma-2}{2\gamma-1}}n^{3\gamma-1}\big(M  \beta(n)^{\tfrac{2}{2\gamma-1}}n\big)^{2-3\gamma} 
    \\ &  \qquad + c M^{1-2\gamma}+ c M^{1-2\gamma} s^{-1}
    \\ &=cM^{1-2\gamma}sn + o(sn)\,.
\end{align*}
%Now let $i,j\in\{1,\dots,m\}\cup O_M$ and $K_\gamma:=\lfloor\tfrac{1-\gamma}{2\gamma-1}\big\rfloor$. We write $i\overset{O_M}{\longleftrightarrow}j$ for the event, that there exists an $l\in\{1,\dots, K_\gamma\}$ and distinct vertices $y_1,\dots,y_{l-1}\in O_M$ and $x_1,\dots,x_{l}\in\{sn+1,\dots,n\}$ such that $i \sim x_0$, $j\sim x_{l}$ and $x_{k} \sim y_{k} \sim x_{k+1}$ for all $k\in\{1,\dots,l-1\}$. 
Having thus proved \eqref{3.17} we get, for any $\varepsilon,\delta\in(0,1)$ an $M_0=M_0(\varepsilon,\delta)$ such that for $M\ge M_0$ and all $n$ large enough,
$$\#\mathcal C(C_1(\mathcal G^{(2)}_{n, \mu})) \le  \sum_{i\in C_1(\mathcal G^{(2)}_{n,M})} \text{outdeg}(i) + \varepsilon sn .$$
The second claim follows from Lemma~\ref{neigh_asymp}.
\end{proof}

By Lemma~\ref{uppersuper} we have, for every  $\varepsilon,\delta\in(0,1)$ and all $M$ large enough
\begin{align*}
     \limsup_{n\to\infty} \frac{1}{sn}\#\big\{x\in\{1,\dots,n\}\,:\, x \sim C_1(\mathcal G_{n,M}^{(2)}) \}  
     \leq
     M^{1-\gamma} \int_0^1 f(x) \rho(\kappa_M , x)\,\text{d}x
     +\varepsilon\,,
\end{align*} 
with probability at least $1-\delta$. By Lemma \ref{mu} the right-hand side converges to $\theta$ as $M\to\infty$. We have thus seen that the size of this component matches the lower bound arbitrarily closely. Moreover, asymptotically almost all the vertices in this component are direct neighbours of the giant in~$\mathcal G_{n,M}^{_{(2)}}$. \smallskip

 It only remains to  prove that the thus constructed component of $\mathscr G_n$
 is indeed the largest. To this end, let us first establish that any sufficiently large component has to intersect the set $\{1,\dots,m\}$.
\begin{lemma}\label{giant}
    Let $(m_n)_{n\in\mathbb N}$ be a sequence of integers such that $m_n\le n$ for all $n$ and $m_n\to\infty$. Then, with high probability,  the largest connected component which does not intersect $\{1,\ldots,m_n\}$ has size at most $\beta(n)^{-1}(n/m_n)^{1-\gamma}$.
\end{lemma}

\begin{proof}
We have for $y>m_n$
\begin{align*}
    \log\mathbb{P}(y \not\sim\{1,\dots,m_n\})\le \sum_{x=1}^{m_n} -\beta(n)x^{-\gamma} n^{\gamma-1} \le \frac{\beta(n)n^{\gamma-1} -\beta(n)n^{\gamma-1}m_n^{1-\gamma}}{1-\gamma} \;.
\end{align*}
We now condition on the $\sigma$-algebra $\mathscr F(m_n,n)$ given by the graph restricted to
$\{m_n+1,\ldots,n\}$. Then, by a union bound over all components $\mathcal{C}$ in this graph,
\begin{align*}
\mathbb{P}\Big( \exists \mathcal{C}\subset\{m_n+1,\dots,n\}\,:\, &{|\mathcal{C}|\ge k\text{ and } \mathcal{C}\not\sim \,\{1,\dots,m_n\}}\Big) \\
& \le \mathbb E\bigg[ \sum_{\mathcal{C}\subset\{m_n+1,\dots,n\}\atop |\mathcal{C}|\ge k}  \mathbb{P}\left( \mathcal{C}\not\sim \,\{1,\dots,m_n\}\big|
\mathscr F(m_n,n)\right)\bigg].
\end{align*}
The summed probabilities are bounded uniformly in $\mathcal C$ and,
as there can be at most $n/k$ components of size at least $k$, we can continue, for some $c>0$, by
\begin{align*}
\phantom{\mathbb{P}\Big( \exists \mathcal{C}\subset} &
& \le \frac{n}{k}\,\exp\left(c\,k(\beta(n)n^{\gamma-1} -\beta(n)n^{\gamma-1}m_n^{1-\gamma}\right).
\end{align*}
The last term goes to 0 if $k\beta(n)n^{\gamma-1}m_n^{1-\gamma}\to \infty$ and this implies that 
$$ \mathbb{P}\Big( \exists \mathcal{C}\subset\{1,\dots,n\}\,:\, {|\mathcal{C}|\ge \beta(n)^{-1}\big( \frac{n}{m_n} \big)^{1-\gamma}\text{ and }\atop \mathcal{C}\not\sim \,\{1,\dots,m_n\}}\Big) \to 0 \,,$$
which completes the proof.
\end{proof}

Applying this lemma with $m_n=\lceil M\beta(n)^{\alpha}n\rceil $ yields that any component with vertices only in $\{m_n+1,\dots,n\}$ has size at most $\beta(n)^{-\alpha/2}$, which is smaller than $sn$ by the assumption that $\beta(n)\gg n^{1/2-\gamma}$. Thus any sufficiently large component has to intersect the set $\{1,\dots,m_n\}$ and therefore also has to contain a component~$K$ of~\smash{$\mathcal G^{_{(2)}}_{n,M}$}. 
We now complete the proof of (I) in Theorem~\ref{main} by showing that this component is small relative to $sn$ if
$K$ is not the largest component of~$\mathcal G^{_{(2)}}_{n,M}$.
\smallskip

\begin{lemma}\label{unique}
For all $\varepsilon, \delta>0$ there exists $M_0=M_0(\varepsilon, \delta)$ such that,
for all $M\ge M_0$,
$$
\liminf_{n\to\infty}
\mathbb P\Big( \#\mathcal C(K) < \varepsilon s n
\text{ for all components $K\not=C_1(\mathcal G^{(2)}_{n,M})$} \Big)>1-\delta.
$$
\end{lemma}

\begin{proof} As before, Lemmas \ref{pathlength}, \ref{compofm} and \ref{orderm} 
imply that, for any $\varepsilon,\delta \in (0,1)$ and sufficiently large $M,n$  such that, with probability at least $1-\delta$, for all connected components $K\subset \mathcal G^{_{(2)}}_{n,M} $, 
\begin{align*}
     \#\mathcal C(K) \le &\,  \sum_{i\in K} \text{outdeg}(i) +\varepsilon sn.
\end{align*}
{Hence, for $\varepsilon_2>\varepsilon_1>0$, $\delta\in(0,1)$ we find $M_0$ such that, for any 
$M\ge M_0$,}
\begin{align}
 \mathbb P \Big(  \max_{{K\in\mathcal{G}^{(2)}_{n,M}:}\atop{K\neq C_1(\mathcal{G}^{(2)}_{n,M})}} & \#\mathcal C(K)>\varepsilon_2 sn\Big) 
 \notag\\
& \leq   \mathbb P \Big(  \max_{{K\in\mathcal{G}^{(2)}_{n,M}:}\atop{K\neq C_1(\mathcal{G}^{(2)}_{n,M})}}\sum_{i\in K} \text{outdeg}(i)> (\varepsilon_2-\varepsilon_1) sn\Big)  + \delta. \label{schritt1}
\end{align}
Recall that $\mathcal G^{(2)}_{n,M}$ is an inhomogeneous random graph with $m'=\lceil M\beta(n)^{\alpha}n\rceil $ vertices and kernel $\kappa_M$ given in (\ref{kappamu}). This kernel satisfies $||T_{\kappa_M}||>1$ and $\inf_{x,y}\kappa_M(x,y)>0$. Hence, by \cite[Theorem 3.12, (ii)]{rjb} there exists a  $\nu=\nu(M)>0$ such that $\#C_2(\mathcal G^{(2)}_{n,M})\le \nu \log m'$ with high probability. Hence
\begin{align}
 \mathbb P \Big(  & \max_{{K\in\mathcal{G}^{(2)}_{n,M}:}\atop{K\neq C_1(\mathcal{G}^{(2)}_{n,M})}}\sum_{i\in K} \text{outdeg}(i)> (\varepsilon_2-\varepsilon_1) sn\Big) \notag \\
& \leq 
  \mathbb P \Big(\max_{{K\subset\{1,\dots,m'\}:}\atop{ \#K\le \nu \log m'}}\sum_{i\in K} \text{outdeg}(i) >(\varepsilon_2-\varepsilon_1) sn\Big) \notag\\ 
& \le   (\nu \log m') \binom{m'}{\nu \log m'} \max_{K\subset\{1,\dots,m'\}\atop \#K\le \nu \log m'} \mathbb P \Big(\sum_{i\in K} \text{outdeg}(i)>(\varepsilon_2-\varepsilon_1) sn\Big).
\label{schritt2}
\end{align}
Furthermore, for the exponential moments of the outdegrees we have
\begin{align*}
     \mathbb{E}\big[e^{\text{outdeg}(i)}\big]=&\,\prod_{j>i}^n  \mathbb{E}\big[e^{ \mathbbm{1}_{\{j\sim i\}}}\big]=\prod_{j>i}^n  \big(1+(e-1)\beta(n)j^{\gamma-1} i^{-\gamma}\big)
      \\ =&\,\exp\Big((e-1)\frac{1}{\gamma}\beta(n)n^{\gamma} i^{-\gamma}\big(1+  o(1)\big) \Big).
\end{align*}
Thus by the exponential Markov inequality and the above 
we have for the sum over the outdegrees
\begin{align*}
   \mathbb P \Big(\sum_{i\in K} \text{outdeg}(i)>\varepsilon sn\Big)
     \le & \,e^{-\varepsilon sn}\prod_{i\in K}\mathbb E\big[e^{\text{outdeg}(i)}\big]
     \\ \le & \,\exp\big(-\varepsilon sn+ c(\#K)n^\gamma \beta(n)\big)\,.
\end{align*}
Combining this with \eqref{schritt1}, \eqref{schritt2} we get
\begin{align*}
    \limsup_{n\to\infty} \mathbb P \Big( & \max_{{K\in\mathcal{G}^{(2)}_{n,M}}\,{K\neq C_1(\mathcal{G}^{(2)}_{n,M})}}\#\mathcal C(K)>\varepsilon_2 sn\Big)
    \\  &  \le \exp\big( \limsup_{n\to\infty} c \log^2(m')-(\varepsilon_2-\varepsilon_1) sn+c\log(m')n^\gamma \beta(n)\big) +\delta\,,
\end{align*}
{where $c>0$ is a positive constant. As we have $sn\gg n^\gamma \beta(n)\log(m')$ and $sn\gg \log^2(m')$ the limsup in the exponent is $-\infty$. This shows that the probability that
$\#\mathcal C(K) > \varepsilon_2 s n$ for some component $K\not=C_1(\mathcal G^{_{(2)}}_{n,M})$ can be made smaller than $\delta$, as required to complete the proof.} 
\end{proof}

This completes the proof of Theorem~\ref{main} in the tiny giant regime~(I).

\subsection{In the critical window} \label{intermed}

In the critical window the core from Lemma \ref{lem1} has size of constant order and splits into random components. We first prove that for sufficiently large $m$, independent of $n$, the largest connected component in $\mathscr G_n$ intersects $\{1,\ldots,m\}$ with probability arbitrarily close to one.\smallskip
% is among them and thus $\#C_1$, appropriately rescaled, has a random limit.

Abbreviate  $\varphi_n:=\beta(n)n^{\gamma-\tfrac12}\to \varphi$.
For the first part of this section we allow $\varphi\ge0$ and later on specify when we only consider $\varphi>0$. The main technical ingredient of the proof is the following upper bound on the expected number of paths ending in a fixed vertex $i$ and running above a threshold $m$, which here is a sufficiently large but otherwise arbitrary integer. 

\begin{lemma}
  Let $X_l^{(i)}$ be the number of paths $x_0x_1\cdots x_l$ of length $l$ such that $x_l=i$ and $x_k>m$ for all $k=0,1,\ldots,l-1$. Then
  \begin{align}\label{esti}
		\mathbb E\big[X^{(i)}_l\big]&\le \mu_l(x_l) \,,
\end{align}
where
$$ \mu_l(x_l) \le \mathbbm{1}_{\{x_l > m\}}\psi_l\, x_{l}^{\gamma-1}+\phi_l\, x_{l}^{-\gamma}$$
and 
\begin{equation} \label{intermedrec}
    \begin{aligned}
    \psi_{l} = &\,O\big(\beta(n)^{l}m^{(1-2\gamma)\lfloor \tfrac{l}2 \rfloor}n^{\gamma + (2\gamma-1)(\lfloor \tfrac{l}2 \rfloor-1)}\log(n)^{\mathbbm{1}_{\{l \text{ odd}\}}}\big) \\
    \phi_{l} = &\,O\big(\beta(n)^{l}m^{(1-2\gamma)\lfloor \tfrac{l-1}2 \rfloor}n^{\gamma + (2\gamma-1)\lfloor \tfrac{l-1}2 \rfloor}\log(n)^{\mathbbm{1}_{\{l \text{ even}\}}}\big)\;.
\end{aligned}
\end{equation} 
\end{lemma}

\begin{proof}
Equation~\eqref{esti} holds when $(\mu_k)$ is given by the recursion (\ref{recurmu}).
The starting value is given by 
\begin{align*}
    \mu_1(x_1)= & \sum_{m<x_{0}\le n} \beta(n)(x_0 \vee x_{1})^{\gamma-1}(x_0 \wedge x_{1})^{-\gamma}
		\\ \le & \mathbbm{1}_{\{x_1 > m\}}\underbrace{\frac{\beta(n) n^{1-\gamma}}{1-\gamma}}_{=:\psi_1} x_{1}^{\gamma-1}+\underbrace{\frac{\beta(n) n^{\gamma}}{\gamma}}_{=:\phi_1} x_{1}^{-\gamma} \;.
\end{align*}
Using Lemma \ref{1.3} all further values can be bounded by 
$$ \mu_k(x_k) \le \mathbbm{1}_{\{x_k > m\}}\psi_k\, x_{k}^{\gamma-1}+\phi_k\, x_{k}^{-\gamma} \,.$$
where $(\psi_k), (\phi_k)$ are recursively defined by 
\begin{align*}
    \psi_{k+1} = &\,\beta(n)\big(\phi_k \frac{m^{1-2\gamma}}{2\gamma-1}+\psi_k\log(n)\big), \\
    \phi_{k+1} = &\,\beta(n)\big(\phi_k\log(n) +\psi_k\frac{n^{2\gamma-1}}{2\gamma-1}\big)\;.
\end{align*}
Using this one can prove the claimed asymptotics by induction.
\end{proof}

%If we can prove that the largest connected component also has to intersect the set $\{1,\dots,m\}$ we are done. {\color{red}Done with what and why?} 
%\noindent 
We use the notation superscript \glqq $>m$\grqq \;to restrict a previously introduced variable/set to the graph on $\{m+1,\dots,n\}$. Writing $C_{i}$ for the $i$-th largest component, we can now bound the size of $C_{1}^{>m}$ from above.

\begin{lemma}\label{lcc}
For any $\varepsilon>0$ and $m$ sufficiently large,
$$ \max_{n\in\mathbb N}
    \mathbb{P}\big(\#C_{1}^{>m}\ge \beta(n)n^\gamma \big) \le \varepsilon \,.$$
\end{lemma}

\begin{proof}
 Observe that
\begin{equation}
    \begin{aligned}\label{cmn}
    \#C_{1}^{>m}  & \le \Big[\sum_{j=1}^n \big( \#C^{>m}_{j}\big)^2 \Big]^{\frac12} =
   \Big[\sum_{j=1}^n\sum_{i\in C_{j}^{>m}} \#C^{>m}(i) \Big]^{\frac12}\\ &   = \Big[\sum_{i={m}+1}^n \#C^{>m}(i) \Big]^{\frac12}.
\end{aligned}
\end{equation}
We can use our bound on the number of paths running above $m$ to get
\begin{align*}
   \mathbb E   \bigg[\sum_{i={m+1}}^n \#C^{>m}(i) \bigg] & \le  \sum_{l\ge1}  \sum_{i=m+1}^{n} \mathbb E\big[X^{(i)}_l\big] 
   \le   \sum_{l\ge1} \frac{\phi_l}{1-\gamma} n^{1-\gamma} + \frac{\psi_l}{\gamma} n^{\gamma} 
    \\ & \le c\beta(n)n+c(\beta(n)n^\gamma)^2m^{1-2\gamma}+cn^{2\gamma}\sum_{l\ge3} \beta(n)^l m^{(\tfrac12-\gamma)l} n^{(2\gamma-1)(\tfrac{l}{2}-1)}
    \\ & = c\beta(n)n+c\,n \varphi_n^2m^{1-2\gamma}+cn \sum_{l\ge3} \big(\varphi_n m^{\tfrac12-\gamma}\big)^l 
    \\ & \le 3 cn\varphi_n^2m^{1-2\gamma}\,,
\end{align*}
where we used that $\phi_ln^{1-\gamma} \le \psi_l n^{\gamma} $ for all $l\ge1$. Therefore by Markov's inequality, (\ref{cmn}) and Jensen's inequality we have, for any $\varepsilon>0$ and $m$ sufficiently large,
\begin{equation}
\begin{aligned}
    \mathbb{P}\big(\#C_{(1)}^{>m}\ge \beta(n)n^\gamma \big) \le &\,\frac{1}{\beta(n)n^\gamma }\Big[c\,n\varphi_n^2m^{1-2\gamma}  \Big]^{1/2}  
     \le \,c\,m^{\tfrac12-\gamma}  \le \varepsilon,
\end{aligned}
\end{equation}
as claimed.
\end{proof}

%\noindent 
Thus the largest connected component restricted to the graph on $\{m+1,\dots,n\}$ is smaller than the largest degree in $\mathscr G_n$ and therefore the largest connected component in $\mathscr G_n$ has to intersect the set $\{1,\dots,m\}$. 
\smallskip

%\noindent 
We now investigate the connection structure of the vertices in $\{1,\dots,m\}$. Note that with high probability they are not connected by single edges. Indeed, as the connection probability is bounded by $\beta(n)$ and there are less than $m^2$ pairs of vertices, the probability of at least one edge existing goes to zero. We therefore look at the core graph as defined Definition~\ref{def1} with $sn=m+1$. Let 
$$X_{ij}:=\#\{x\in\{m+1,\dots,n\}:\; i\sim x \sim j \}.$$ 
We use a Poisson approximation argument to show convergence in total variation of these random variables to independent Poisson variables.

\begin{lemma}\label{indep}
    Let $(P_{ij})_{1\le i,j\le m}$ be independent Poisson random variables with parameter $\lambda_{ij}=\frac{\varphi^2}{2\gamma-1} i^{-\gamma}j^{-\gamma}$. Then, for every fixed $m\in\mathbb N$,
    $$ (X_{ij})_{1\le i,j\le m} \overset{tv}{\longrightarrow}(P_{ij})_{1\le i,j\le m}\,.$$
    In particular the $(X_{ij})_{1\le i,j\le m}$ are asymptotically independent.
\end{lemma}

\begin{proof}
    We focus on $X_{12}$ and $X_{13}$, the proof for the full family $(X_{ij})_{1\le i,j\le m}$ is analogous, but with more cumbersome notation. We use \cite[Corollary 2.2]{AHH} which can be stated as follows\footnote{The cited reference uses convergence in distribution but as all involved random variables take only countably many values, convergence in distribution and in total variation are equivalent.}: 
    Let \smash{$X_{12}^{(p)} \sim \text{Bin}(X_{12},p)$}, \smash{$X_{13}^{(q)} \sim \text{Bin}(X_{13},q)$} and suppose that 
    \begin{align} \label{whatistoshow}
         X_{12}^{(p)}+X_{13}^{(q)} \overset{tv}{\longrightarrow} \text{Pois}(p\lambda_{12} + q \lambda_{13})\,.
    \end{align}
   Then $(X_{12}, X_{13})$ converges in total variation to two independent Poisson random variables with means $\lambda_{12}$ and $\lambda_{13}$ respectively. 
   \smallskip
   
   To prove (\ref{whatistoshow}) we need results from \cite{BHJ} and first embed our setup into their framework. We equip every vertex $x$ with two independent random variables $J_x^2 \sim \text{Ber}(p)$ and $J_x^3 \sim \text{Ber}(q)$. Then
   \begin{align*}
       X_{12}^{(p)}+X_{13}^{(q)} = \sum_{x\in\{m+1,\dots,n \}\atop k\in\{2,3\}} J_x^k \mathbbm{1}_{1\sim x} \mathbbm{1}_{x\sim k} = \sum_{\alpha\in\Gamma} I_\alpha \,,
   \end{align*}
   where $\Gamma := \bigcup_{x=m+1}^n \{(1,x,2), (1,x,3)\}$ and $I_{(1,x,k)}:=J_x^k \mathbbm{1}_{1\sim x} \mathbbm{1}_{x\sim k}$. 
   %are dependent indicators. 
   In particular, for any $\alpha=(1,x,k)\in\Gamma$ the set of all $\beta$ such that $I_\alpha$ and $I_\beta$ are independent is given by $(\Gamma\setminus \{\alpha\})\setminus \{(1,x,k^c)\}$ as any other index uses a different middle vertex (where $k^c:=3$ if $k=2$ and $k^c:=2$ otherwise). Let $\pi_\alpha = \mathbb{P}(I_\alpha =1)$ and $\lambda=\mathbb E [\sum_{\alpha\in\Gamma} I_\alpha]$. By  \cite[Corollary 2.C.5]{BHJ},
   \begin{align*}
       d_{tv}\Big(\mathcal{L}\big(\sum_{\alpha\in\Gamma} I_\alpha\big), \text{Pois}(\lambda) \Big) \le \frac{1-e^{-\lambda}}{\lambda}\Big(\sum_{\alpha\in\Gamma}\pi_\alpha^2 +\sum_{\alpha=(1,x,k)\in\Gamma}\sum_{\beta=(1,x,k^c)}\mathbb E[ I_\alpha I_\beta ]+\pi_\alpha\pi_\beta\Big).
   \end{align*}
   We have 
   \begin{align*}
       \pi_{(1,x,k)} = \left\{\begin{matrix} p \beta(n)^2 1^{-\gamma}2^{-\gamma} x^{2(\gamma-1)} & \text{if } k=2,\\ q \beta(n)^2 1^{-\gamma}3^{-\gamma} x^{2(\gamma-1)}& \text{if } k=3,   \end{matrix} \right.
   \end{align*}
and 
\begin{align*}
    \mathbb E[ I_{(1,x,k)} I_{(1,x,k^c)} ]=pq \mathbb E\big[ \mathbbm{1}_{1\sim x} \mathbbm{1}_{x\sim 2}\mathbbm{1}_{x\sim 3}\big]=pq  \beta(n)^3 6^{-\gamma} x^{3(\gamma-1)}\,.
\end{align*}
Hence, by the above for two constants $c_1,c_2>0$,
\begin{align*}
     d_{tv}\Big(\mathcal{L}\big(\sum_{\alpha\in\Gamma} I_\alpha\big), \text{Pois}(\lambda) \Big) \le & \frac{1}{\lambda}\Big(c_1\beta(n)^4\sum_{x=m+1}^nx^{4(\gamma-1)} +c_2\beta(n)^3\sum_{x=m+1}^nx^{3(\gamma-1)}\Big)
      \\ \le & \frac{1}{\lambda}\big( c_1 \varphi_n^4 n^{-1}+c_2 \varphi_n^3 n^{-1/2}\big) \to 0 \;.
\end{align*}
Lastly we compute 
\begin{align*}
    \lambda=\mathbb E \Big[\sum_{\alpha\in\Gamma} I_\alpha\Big]= p  \sum_{x=m+1}^n \pi_{(1,x,2)} + q  \sum_{x=m+1}^n \pi_{(1,x,3)} \to p\tfrac{\varphi^2}{2\gamma-1}  1^{-\gamma}2^{-\gamma} + q \tfrac{\varphi^2}{2\gamma-1} 1^{-\gamma}3^{-\gamma} \,.
\end{align*}
\end{proof}

%\noindent 
We now define $\mathscr{G}_{\text{NR}}=\mathscr{G}_{\text{NR}}\big((\lambda_{ij})_{i,j\in\mathbb{N}}\big)$ as the random graph with vertex set $\mathbb N$, where two vertices have an $\text{Pois}(\lambda_{ij})$ amount of edges between them, independent of all other vertices. Moreover let $\mathcal{C}(i)$ be the component of $i$ in $\mathscr{G}_{\text{NR}}$. By the above lemma we can couple  our core graph with $\mathscr{G}_{\text{NR}}$, such that they 
agree up to the first~$m$ vertices with probability going to one as $n\to\infty$.
Moreover we define
$$\text{outdeg}(i)=\sum_{j>i}\mathbbm{1}_{\{i\sim j\}}\;\text{ and }\;\text{indeg}(i)=\sum_{j<i}\mathbbm{1}_{\{i\sim j\}}\,.$$
The following lemma shows that the main contributors to $\# C(i)$ are the outdegrees in $\mathscr{G}_n$ of the vertices in the component of $i$ in $\mathscr{G}_{\text{NR}}$.

\begin{lemma}\label{3.9}
If $\varphi_n=\beta(n)n^{\gamma -1/2} \to \varphi >0$, then 
for every fixed $i\in\mathbb N$,  
\begin{align}\label{pwc}
    \frac{\# C(i)}{n^\gamma \beta(n)}\overset{d}{\longrightarrow} \sum_{j\in\mathcal C(i)} \frac{1}{\gamma}\, j^{-\gamma} \;.
\end{align}
\end{lemma}

\begin{proof}
For every fixed $j$ and $m$ we have
\begin{align*}
    \mathbb{E}\big[\text{outdeg}^{>m}(j)\big]=\sum_{l=m+1}^n \beta(n) l^{\gamma-1} j^{-\gamma} = \frac1\gamma \beta(n) n^\gamma j^{-\gamma} +o(1)\;.
\end{align*}
Furthermore, for the variance we have, similar to (\ref{varj}), that $\Var(\text{outdeg}^{>m}(j))\le\frac{1}{\gamma} \beta(n) n^\gamma j^{-\gamma} $.
Thus by Chebyshev's inequality
\begin{equation} \label{cheby}
     \frac1{n^\gamma \beta(n)}\text{outdeg}^{>m}(j) \overset{\mathbb P}{\longrightarrow} \frac{1}{\gamma}\, j^{-\gamma} \;.
\end{equation}
We fix $i$ and denote by $C'(i)$ the component of vertex $i$ in the core graph. Note that
\begin{align}\label{cge1}
  \# C(i) \geq \sum_{j\in C'(i)} \text{outdeg}^{>m}(j) - \# \bigg\{ j\in\{m+1,\ldots,n\}{ \exists k\not=l\in C'(i)\atop \text{ with } k,l\sim j}\bigg\}.  
\end{align}
The subtracted set is not empty, as every edge $\{k,l\}$ in the core graph causes the middle vertex to be an element of this set. We now argue that this set is negligible. Note that its expectation can be bounded from above by 
\begin{align}\label{cge2}
    \mathbb E\bigg[\sum_{j=m+1}^n\sum_{1\le k < l \le m} \mathbbm{1}_{\{l\sim j \sim k\}}\bigg ] \le \sum_{j=m+1}^n \beta(n)^2 j^{2\gamma-2 }m^2 \le \tfrac{1}{2\gamma-1}\beta(n)^2 n^{2\gamma-1 } m^2. % \le c m^2\,.
\end{align}
Hence, scaled by $n^\gamma \beta(n)$, the second term in the inequality above goes to 0 in probability. This implies that, for fixed $i\le m$ and $\varepsilon>0$,
$$\lim_{n\to\infty} \mathbb P\Big( \frac{\# C(i)}{n^\gamma \beta(n)} \ge 
\sum_{j\in C'(i)}  \frac1\gamma j^{-\gamma}-\varepsilon \Big) = 1,$$
where $C'(i)$ agrees with the component  of $i$ in the restriction of
$\mathscr{G}_{\text{NR}}$ to the first $m$ vertices. The lower bound in (\ref{pwc}) follows as $m$ was arbitrary.\smallskip

To prove the upper bound in (\ref{pwc}) note that
\begin{align}\label{cle}
   \# C(i) \leq \sum_{j\in C'(i)} \text{outdeg}^{>m}(j) + \# \bigg\{ k\in\{m+1,\ldots,n\} \colon {k \leftrightarrow i, \,\atop k \not\sim j \,\,\forall j\in C'(i)} \bigg\}. 
\end{align}
It remains to show that the added set on the right is negligible, i.e.\ that
vertices with a distance greater than two 
to~$i$, which are not a nearest neighbour of some $j\in\{1,\dots,m\}$, are negligible. Denote by $C^{>m}_{\ge2}(i)$ the set of said vertices. Then we prove for any $\varepsilon>0$, that
\begin{align} \label{negl}
   \lim_{m\to \infty} \lim_{n\to \infty} \mathbb{P}\Big(\frac{1}{n^\gamma\beta(n)}\#C^{>m}_{\ge 2}(i) > \varepsilon \Big) =0 \,.
\end{align}
Recall that $X^{(i)}_l$ counts the paths running of length $l$ running above the threshold $m$ and ending in $i$ and observe that \smash{$\#C^{>m}_{\ge 2}(i) \le \sum_{l\ge2} X^{(i)}_l$}. By (\ref{intermedrec}) we can estimate the expectation of the latter by 
\begin{align*}
    \mathbb{E}\Big[\sum_{l\ge2} X^{(i)}_l\Big] \le &  \sum_{l\ge2} \mu_l(i) \le  \sum_{l\ge2}  \phi_l i^{-\gamma} 
    \\ \le & \,c \sum_{l\ge2} \varphi_n^l m^{(\tfrac12-\gamma)l} \sqrt{n} \le c \,m^{1-2\gamma} n^{\gamma} \beta(n) \;.
\end{align*}
Now (\ref{negl}) follows from Markov's inequality. Together with
(\ref{cheby})  and Lemma \ref{indep} this implies (\ref{pwc}) and hence the lemma.
\end{proof}

\noindent We now complete the proof of Theorem~\ref{main} $(I\!I)$.

\begin{lemma}\label{conclusion}
    We have
    \begin{align*}
    \frac{\# C_1(\mathscr{G}_n)}{n^\gamma \beta(n)}  \overset{d}{\longrightarrow} \sup_{i\in\mathbb N}\bigg\{\sum_{j\in\mathcal C(i)} \frac{1}{\gamma}\, j^{-\gamma}\bigg\}\;.
\end{align*}
\end{lemma}

\begin{proof}
Let us first show that the limiting object on the right hand side in the lemma is well-defined. 
We identify components by its smallest index vertex, i.e. 
\begin{align*}
    \mathcal C_{\le}(i):=\left\{\begin{matrix}
          \mathcal C(i)\,, & \text{if } \,i=\min\{j\,:\,j\in \mathcal C(i)\}  \\
            \emptyset\,, & \text{ otherwise,}
    \end{matrix}\right.
\end{align*} similarly we define $C_{\le}(i)$. 
Then it suffices to show that
$$\mathbb E\bigg[\sum_{i=1}^\infty \Big(\sum_{j\in \mathcal C_{\le}(i)}\frac{1}{\gamma}\, j^{-\gamma}\Big)^2\bigg]< \infty \,.$$
The left hand side is equal to
\begin{align*}
   \mathbb E\bigg[ & \sum_{i=1}^\infty\Big(\sum_{j\in \mathcal C_{\le}(i)}\frac{1}{\gamma^2}\, j^{-2\gamma}+\sum_{j,k\in \mathcal C_{\le}(i)}\frac{1}{\gamma^2}\, (jk)^{-\gamma}\Big)\bigg] \\
   & \le \, \frac{1}{\gamma^2}\sum_{i=1}^\infty i^{-2\gamma}+ \frac{1}{\gamma^2}  \sum_{j,k=1}^\infty (jk)^{-\gamma} \mathbb{P}\big( j\leftrightarrow  k\text{ in } \mathscr G _{\text{NR}}\big)\;.
\end{align*}
The first term is finite as $\gamma>\tfrac12$. For the second term, note that
$$\mathbb{P}\big( i\sim   j\text{ in } \mathscr G _\text{NR}\big)\le \lambda_{ij}=\frac{\varphi^2}{2\gamma-1}(ij)^{-\gamma} \le \lambda_c \underbrace{ (i \wedge j)^{\gamma-1}(i \vee j)^{-\gamma}}_{=:h(i,j)}\;,$$
where the last inequality follows for all, but finitely many, tuples $(i,j)$ and %$\lambda_c$ is given by
\begin{align*}
    \lambda_c = & \bigg(\int_0^\infty \frac{h(1,y)}{\sqrt{y}}\,\text{d}y\bigg)^{-1}=\bigg(\Big[\frac{1}{\gamma-\frac12}y^{\gamma-\frac12}\Big]_0^1+\Big[\frac{1}{\frac12-\gamma}y^{\frac12-\gamma}\Big]_1^\infty\bigg)^{-1}
    =\frac{\gamma}{2}-\frac{1}{4}\,.
\end{align*}
We define $\mathscr G _{\text{DK}}$ as the random graph on the vertex set $\mathbb N$, where we draw an edge between to vertices $i$ and $j$ with probability $\lambda_c h(i,j)\wedge 1$, independent of all other edges. This random graph is inhomogeneous of rank $-1$, i.e. $h(ti,tj)=t^{-1}h(i,j)$ and $\lambda_c$ is chosen such that the graph is at criticality. For this particular class of random graphs it was shown in \cite{YZ}, that 
$$\mathbb{P}\big( i\leftrightarrow     j\text{ in } \mathscr G _{\text{DK}}\big)\le c \,\frac{\log(i\wedge j)}{\sqrt{ij}}\;.$$
We can couple $\mathscr G _{\text{NR}}$ and $\mathscr G _{\text{DK}}$ such that every edge that exists in $\mathscr G _{\text{NR}}$ also exists in~$\mathscr G _{\text{DK}}$, except for finitely many edges.
Putting everything together, we arrive at
\begin{align*}
    \mathbb E\bigg[\sum_{i=1}^\infty \Big(\sum_{j\in \mathcal C_{\le}(i)}\frac{1}{\gamma}\, j^{-\gamma}\Big)^2\bigg] \le c_1 + c_2\sum_{j,k=1}^\infty (jk)^{-\gamma-\tfrac12} \log(j \wedge k) <\infty \,.
\end{align*}
To prove convergence we want to argue by Lemma \ref{3.9} and Cramer-Wold, that for every fixed $m\in\mathbb N$
\begin{align}\label{finitem}
   \bigg( \frac{C_{\le}(i)}{n^\gamma \beta(n)}\bigg)_{1\le i \le m} \overset{d}{\longrightarrow} \bigg(\sum_{j\in\mathcal C_{\le}(i)} \frac{1}{\gamma}\, j^{-\gamma}\bigg)_{1\le i \le m} \,.
\end{align}
Therefore let $(t_1,\dots,t_m)\in\mathbb R^m$ be arbitrary, recall that $C'(i)$ is the component of $i$ in the core graph. Similar to (\ref{cge1}) we have for all $M$ sufficiently large
\begin{align*}
 \sum_{i=1}^m t_i \frac{C_{\le}(i)}{n^\gamma \beta(n)} \ge  \sum_{i=1}^m  \frac{t_i}{n^\gamma \beta(n)}\bigg[  \sum_{j\in C_{\le}'(i)} \text{outdeg}^{>M}(j) - \# \bigg\{ j>M\,:{ \exists k\not=l\in C'_{\le}(i)\atop \text{ with } k,l\sim j}\bigg\}\bigg].  
\end{align*}
By (\ref{cge2}) the rightmost term goes to zero in probability and leftmost term on the right hand side converges in distribution to the desired object, since the core graph converges in distribution to $\mathscr G_{\text{NR}}$  by Lemma \ref{indep} and the discussion thereafter. The matching upper bound follows by (\ref{cle}) and (\ref{negl}). Thus (\ref{finitem}) follows.
Now note that by the remark following Lemma~\ref{lcc}, for every $\varepsilon>0$, there exists an $m$, such that
$$ \frac{C_1(\mathscr{G}_n)}{n^\gamma \beta(n)}=\max_{1\le i \le m}\frac{C_{\le}(i)}{n^\gamma \beta(n)}$$
with probability greater than $1-\varepsilon$. By continuity of the $max$ function we conclude 
$$ \frac{C_1(\mathscr{G}_n)}{n^\gamma \beta(n)}=\max_{1\le i \le m}\frac{C_{\le}(i)}{n^\gamma \beta(n)} \overset{d}{\longrightarrow} \sup_{1\le i \le m}\bigg\{\sum_{j\in\mathcal C_{\le}(i)} \frac{1}{\gamma}\, j^{-\gamma}\bigg\}\;,$$
with probability greater than $1-\varepsilon$. Letting $\varepsilon \to 0$ and thus $m\to\infty$ we therefore obtain
\begin{align*}
    \frac{\# C_1(\mathscr{G}_n)}{n^\gamma \beta(n)}  \overset{d}{\longrightarrow} \sup_{i\in\mathbb N}\bigg\{\sum_{j\in\mathcal C(i)} \frac{1}{\gamma}\, j^{-\gamma}\bigg\}\;,
\end{align*}
using that the right hand side is well defined.
\end{proof}

\subsection{Below the critical window}

By (\ref{cheby}) the degrees of the first $m$ vertices concentrate around their mean and thus by Lemma~\ref{lcc} the largest connected component still has to intersect the set $\{1,\dots,m\}$. Lemma \ref{indep} suggests that the components of the first $m$ vertices are disjoint. The next lemma confirms this.

\begin{lemma}\label{disj}
Suppose $\beta(n)n^{\gamma-\frac12}\to0$. Then, for every $m\in\mathbb N$, the components $C(1),\ldots, C(m)$ in $\mathscr G_n$ are pairwise disjoint with high probability.
\end{lemma}

\begin{proof}
We show that the expected number of paths with start and endpoints in $\{1,\ldots,m\}$ goes to 0. Let $\tilde X_l$ be the number of paths $x_0x_1\cdots x_{l-1}x_l$ of length $l$ satisfying $x_0,x_l\le m$ and $x_k>m$ for all $0<k<l$. We have by Markov's inequality 
\begin{align*}
    \mathbb P ( & \text{components induced by }\{1,\ldots,m\}\text{ are disjoint}) \\ & \ge \mathbb P (\text{no edges in }\{1,\ldots,m\}) \big(1 -  \mathbb P\big (\sum_{l=2}^n \tilde X_l \ge 1\big)\big) \\  &  \ge \exp \big( - \beta(n)m^{2}\big) \Big(1- \mathbb{E}\Big[\sum_{l=2}^n \tilde X_l\Big]\Big)\;.
\end{align*}
To estimate the expectation above, we use the same recursion as in the last section. Recall (\ref{14}) and Lemma \ref{1.3}, here we have for the starting value that
\begin{align*}
    \mu_1(x_1) = \sum_{0<x_0 \le m} \beta(n) x_0^{-\gamma} x_1^{\gamma-1} \le \mathbbm{1}_{\{x_1 > m\}}\frac{\beta(n) m^{(1-\gamma)}}{1-\gamma}\, x_{1}^{\gamma-1} \;.
\end{align*}
Hence we define \smash{$\tilde \psi_1=\tfrac{\beta(n) m^{1-\gamma}}{1-\gamma}$} and $\tilde \phi_1=0$. Now applying Lemma \ref{1.3} with $s_1= s_2 = m n^{-1}$ (thus $s_1n=m$ and \smash{$\log(\tfrac{1}{s_1})< \log(n)$}) we define 
\begin{equation}\label{rec}
    \begin{aligned}
		\tilde \psi_{k+1}&=\tilde \phi_k \beta(n) \frac{m^{(1-2\gamma)}}{2\gamma-1}+\tilde \psi_k \beta(n) \log(n)
		\\ \tilde \phi_{k+1}&=\tilde \phi_k \beta(n) \log(n)+\tilde \psi_k \beta(n) \frac{n^{2 \gamma -1}}{2\gamma-1} \;.
\end{aligned}
\end{equation}
Similar to the last section, one can prove by induction that for all $k\ge 2$
\begin{equation} \label{esti1}
    \begin{aligned}
      \tilde \psi_{k} = &\,O\big(\beta(n)^{k}m^{(1-\gamma)+(1-2\gamma)\lfloor \tfrac{k-1}2 \rfloor}n^{(2\gamma-1)\lfloor \tfrac{k-1}2 \rfloor}\log(n)^{\mathbbm{1}_{\{k \text{ even}\}}}\big) \\
    \tilde \phi_{k} = &\,O\big(\beta(n)^{k}m^{(1-\gamma)+(1-2\gamma)(\lfloor \tfrac{k}2 \rfloor-1)}n^{(2\gamma-1)\lfloor \tfrac{k}2 \rfloor}\log(n)^{\mathbbm{1}_{\{k \text{ odd}\}}}\big)\;.
    \end{aligned}
\end{equation}
By the above estimate and (\ref{14}) it follows for a suitable  $c>0$ that 
\begin{align*}
      \mathbb{E}[X_l]&\le\sum_{0<x_l\le m} \phi_l x_l^{-\gamma} \le c\,m\beta(n)^{l} m^{(\tfrac12-\gamma)l} n^{(\gamma-\tfrac12)l} \;.
\end{align*}
Summing over all $l\ge2$ yields
\begin{align*}
    \sum_{l=2}^n \mathbb{E}[X_l]&\le c \,m \sum_{l=2}^\infty \Big(\beta(n) m^{\tfrac12-\gamma} n^{\gamma-\tfrac12} \Big)^l 
     \le c\,m^{2(1-\gamma)} (\beta(n) n^{\gamma-\tfrac12})^2\,,
\end{align*}
which goes to zero as $n\to\infty$ by assumption. Thus the components of the first $m$ vertices are disjoint. 
\end{proof}

\noindent By (\ref{negl}) the number of vertices in $C(i)$ which are not a nearest neighbour of at least one of the vertices in $\{1,\dots,m\}$ is of smaller order than $n^\gamma \beta(n)$. This implies that
\begin{equation*}\frac{\#C(i)}{n^\gamma \beta(n)}=\frac{\text{outdeg}(i)}{n^\gamma \beta(n)}+o(1)\overset{\mathbb P}{\longrightarrow} \frac{1}{\gamma}i^{-\gamma}\,. 
\end{equation*}
As $m$ is bounded this completes the proof in the star-component regime~(III).
\medskip

%\subsection{Below the critical window: The isolated edges regime}

As the largest degree was the main contributor in the star component regime, it is quite intuitive that this should be the case in the  isolated edges regime  as well. 
%{\color{red}Isn't this even proved in the section above?}{\color{blue}kinda, there is the difference that the degrees dont go to infinity anymore and thus you cannot order them like in the last section. Here you have a finite number of finite random variables and thus you cannot argue that the degree of 1 is the largest anymore.}
Let us first consider the case where $\beta(n)n^\gamma \to 0$ as $n\to\infty$.

\begin{lemma}\label{outdeg}
\begin{itemize}
    \item[(a)]
    Let $P_1,\dots,P_n$ be independent random variables where $P_j$ is Poisson distributed with parameter $\lambda_j=\sum_{i=j+1}^n \beta(n) j^{-\gamma}i^{\gamma-1}$. Then $\emph{outdeg}(1),$ $\dots,\emph{outdeg}(n) $ can be coupled to $P_1,\dots,P_n$, such that
    $$ \mathbb P\big(\exists j\in\{1,\dots,n\}\,:\,\emph{outdeg}(j)\neq P_j\big)\to0\,.$$
\item[(b)]
    Let $\hat P_1,\dots,\hat P_n$ be independent random variables where $\hat P_j$ is Poisson distributed with parameter $\hat \lambda_j=\sum_{i=1}^{j-1} \beta(n) i^{-\gamma}j^{\gamma-1}$. Then $\emph{indeg}(1),\dots,\emph{indeg}(n) $ can be coupled to $\hat P_1,\dots, \hat P_n$, such that
    $$ \mathbb P\big(\exists j\in\{1,\dots,n\}\,:\,\emph{indeg}(j)\neq \hat P_j\big)\to0\,.$$
\end{itemize}    
\end{lemma}

\begin{proof}
We do the proof for the outdegrees, the indegrees follow analogously. By Le Cam's theorem, see e.g.\ \cite[(1.23)]{BHJ}, we have 
\begin{align*} \sum_{j=1}^n \sum_{k=0}^\infty \big| \mathbb P\big(\text{outdeg}(j)=k\big)-\mathbb P\big(P_j = k\big) \big| & \le  \sum_{j=1}^n \sum_{i=j+1}^n \beta(n)^2 j^{-2\gamma}i^{2(\gamma-1)} \\ & \le C \beta(n)^2 n^{2\gamma-1} \to 0 \, . 
\end{align*}\\[-14mm]
\end{proof}

\medskip

Note that the family $\hat P_1,\dots,\hat P_n$ is not independent of the family $ P_1,\dots, P_n$. Also note that $\hat \lambda_j \asymp \beta(n)$ with $\hat \lambda_1 \le \hat \lambda_2 \le \ldots \le \hat \lambda_n$. Let $\varepsilon\in(0,\gamma)$ be arbitrary, then 
\begin{equation}
    \begin{aligned} \label{in}
    \mathbb P \Big(\max_{j=1,..,n^\varepsilon}\hat P_j= 0 & \,,\;\max_{j=n^\varepsilon,..,n}\hat P_j\le1 \Big)=\,\prod_{j=1}^{n^\varepsilon} \mathbb P \big(\hat P_j=0 \big)\prod_{j=n^\varepsilon+1}^n  \mathbb P \big(\hat P_j\le  1 \big) \\\ge\, &e^{-c\beta(n)n^\varepsilon}\, \mathbb P \big(\hat P_n\le  1 \big)^n  \ge  \,e^{-c\beta(n)n^\varepsilon}\, e^{-c\beta(n)n}(1+c\beta(n))^n
    \\ \ge\, & \exp (-c\beta(n)n^\varepsilon- \tfrac{ c^2}{2}\beta(n)^2 n) \to 1\,.
\end{aligned}
\end{equation}
Note that $\lambda_j \sim  \tfrac1\gamma j^{-\gamma}\beta(n) n^{\gamma}$. Hence for the outdegrees we have
\begin{equation}
    \begin{aligned}\label{out}
      \mathbb P \Big(\max_{j=1,..,n} P_j\le  1 \Big) =\,  e^{-\sum_{j=1}^n \lambda_j} \prod_{j=1}^n (1+\lambda_j)  \ge\, \exp\big( -c (\beta(n)n^\gamma)^2\big)\to 1\,.
\end{aligned}
\end{equation}
Thus with high probability the largest connected component can at most be a monotone path, that is $C_1=(y_0,\dots,y_l)$ for some $l\in\mathbb N$ and vertices $y_0<y_1<\ldots<y_l$. Let $Y_l$ be the number of all such paths for a fixed length $l$. Then
\begin{align*}
    \mathbb{E}\big[Y_l\big]\le & \sum_{y_l=1}^n \cdots  %\sum_{y_1=1}^{y_2}
    \sum_{y_0=1}^{y_1} \beta(n)^l y_l^{\gamma-1} y_{l-1}^{-1}\cdots y_{1}^{-1} y_0^{-\gamma} 
     \le  \sum_{y_l=1}^n \cdots  \sum_{y_1=1}^{y_2} \beta(n)^l y_l^{\gamma-1} y_{l-1}^{-1}\cdots y_{2}^{-1} \tfrac{y_1^{-\gamma} }{1-\gamma}
    \\ \le & \Big(\frac{\beta(n)}{1-\gamma}\Big)^l\sum_{y_l=1}^n y_l^{\gamma-1}y_l^{1-\gamma} = \Big(\frac{\beta(n)}{1-\gamma}\Big)^ln\;.
\end{align*}
Thus by Markov's inequality
\begin{equation}\label{eyl}
%    \begin{aligned}
    \mathbb P\big(\exists l\ge2 \text{ such that }Y_l>0\big)\le 
    %\mathbb P\big(\sum_{l=2}^nY_l\ge1 \big) \le 
    \sum_{l=2}^n\mathbb{E}\big[Y_l\big]
    \le \,c\,n \Big(\frac{\beta(n)}{1-\gamma}\Big)^2\to 0\,.
%\end{aligned}
\end{equation}
As the largest connected component has to be a monotone path, but there is no such path with length greater than two, we have shown 
$\mathbb P\big( \#C_1(\mathscr{G}_n)=2\big)\to 1$.\medskip

If $\beta(n)n^\gamma \to c$, then $\lambda_1 \to \tfrac{c}{\gamma}$. Moreover we see that (\ref{in}) still holds and (\ref{out}) can be weakened to: For every $\delta>0$ there exists a  $N=N(\delta)$ such that 
\begin{equation} \label{wout}
\begin{aligned}
     \mathbb P \Big(\max_{j=N,..,n} P_j\le  1 \Big)  \ge \, &  e^{-\sum_{j=N}^n \lambda_j}\prod_{j=N}^n (1+\lambda_j)  \ge\,  \exp\Big(- \sum_{j=N}^n c' \beta(n)^2n^{2\gamma}j^{-2\gamma} \Big)\\\ge\,& \exp\big(- c'' N^{1-2\gamma})\ge 1-\delta\,.
\end{aligned}    
\end{equation}
Given $\delta>0$ and thus $N$, the family $(\max_{j=1,..,N}P_j)_n$ is tight since
\begin{equation} \label{tght}
    \mathbb P \Big(\max_{j=1,..,N} P_j> A\Big) \le \frac1A \sum_{j=1}^N \lambda_j \le \frac{c'}{A}<\delta \,,
\end{equation} 
where the last line follows for $A=A(\delta)$ sufficiently large. 
\pagebreak[3]

We have shown that with probability $1-\delta$ the largest connected component contains only one of the vertices of $\{1,\dots,N\}$ and hence its nearest neighbours. Since all the indegrees are less or equal one and by (\ref{eyl}) every monotone path has at most length one, the largest connected component has a star shape. Hence
$$ \mathbb P \Big(\#C_1(\mathscr{G}_n)\le2\max_{j=1,..,N} P_j+1\Big)\ge 1-\delta \;.$$
With $\delta\to0$ and using that $(\max_{j=1,..,n}P_j)_{n\in\mathbb N}$ is tight, we obtain that $(\#C_1(\mathscr{G}_n))_{n\in\mathbb N}$ is tight as well. 

{ %\footnotesize
}
\end{document}